\documentclass[12pt]{amsart}
\usepackage[margin=1in]{geometry}                
\usepackage{graphicx}
\usepackage{comment}

\usepackage{amsmath}
\usepackage{amssymb}
\usepackage{amsfonts}
\usepackage{amsthm}
\usepackage{mathrsfs}
\usepackage{enumerate}
\usepackage{float}

\usepackage{xcolor}

\usepackage{mathtools}

\usepackage{subcaption}
\usepackage{thmtools, thm-restate}
\usepackage{array}

\declaretheorem{theorem}
\numberwithin{theorem}{section}
\declaretheorem[sibling=theorem]{corollary, lemma, proposition, question, definition, conjecture, example, remark}

\numberwithin{equation}{section}

\newcommand{\zz}{{\mathbb Z}}
\newcommand{\nn}{{\mathbb N}}

\newcommand{\RR}{{\mathbb R}}
\newcommand{\ZZ}{{\mathbb Z}}
\newcommand{\NN}{{\mathbb N}}

\newcommand{\nno}{{\mathbb N}_0}

\newcommand{\Z}{\mathbb{Z}}

\newcommand{\bbA}{\mathbb{A}}
\newcommand{\bbB}{\mathbb{B}}

\newcommand{\bbT}{\mathbb{T}}

\DeclareMathOperator{\Div}{Div}
\DeclareMathOperator{\lcm}{lcm}

\newcommand{\supp}{{\mathrm{supp}}}

\newcommand{\bfT}{\mathbf{T}}

    \newcommand{\calm}{\mathcal{M}}

\newcommand{\calD}{\mathcal{D}}

\newcommand{\cald}{\mathcal{D}}

\newcommand{\capmin}{\hbox{\sf MIN}}
\newcommand{\capexp}{\hbox{\sf EXP}}
\newcommand{\capfib}{\hbox{\sf FIB}}

\begin{document}
\title{Lower bounds for mask polynomials with many cyclotomic divisors}
\author{Gergely Kiss, Izabella {\L}aba, Caleb Marshall, G\'abor Somlai}

\date{\today}

\begin{abstract}
Given a nonempty set $A \subset \mathbb{N}\cup\{0\}$, define the mask polynomial $A(X)=\sum_{a\in A} X^a$. Suppose that there are $s_1,\dots,s_k\in\nn\setminus\{1\}$ such that the cyclotomic polynomials $\Phi_{s_1},\dots,\Phi_{s_k}$ divide $A(X)$. What is the smallest possible size of $A$? For $k=1$, this was answered by Lam and Leung in 2000. Less is known about the case when $k\geq 2$; in particular, one may ask whether (similarly to the $k=1$ case) the optimal configurations have a simple ``fibered" structure on each scale involved. We prove that this is true in a number of special cases, but false in general, even if further strong structural assumptions are added. Results of this type are expected to have a broad range of applications, including
Favard length of product Cantor sets, Fuglede's spectral set conjecture, and the Coven-Meyerowitz conjecture on integer tilings.

\end{abstract}


\maketitle

\section{Introduction}\label{sec-intro}

\subsection{Overview}
Let $A$ be a nonempty set in $ \nno:= \mathbb{N}\cup\{0\}$.
We define the \textit{mask polynomial} of $A$ to be
$$
A(X) : = \sum_{a } X^a.
$$

Recall that the $s$-th cyclotomic
polynomial $\Phi_s(X)$ is the unique monic irreducible polynomial  in $\mathbb{Q}[x]$ whose roots are the
primitive $s$-th roots of unity. Equivalently, $\Phi_s$ can be computed from
the identity $X^n-1=\prod_{s \mid n}\Phi_s(X)$. In particular, if $p$ is a prime number, then
\begin{equation}\label{e-primecyclo}
\Phi_p(X)=\frac{X^p-1}{X-1}=1+X+X^2+\dots + X^{p-1}.
\end{equation}

We are interested in lower bounds on the size of sets whose mask polynomials have prescribed cyclotomic divisors. A classic result of Lam and Leung
 \cite{LamLeung} implies that if $\Phi_s(X)\mid A(X)$ for some $s>1$, we must have
\begin{equation}\label{e-lamleung}
|A| \geq \min \{p:\ p\mid s,\ p \hbox{ is prime}\}.
\end{equation}
We consider the following more general question. Let $S\subset \NN \setminus\{1\}$ be nonempty, and let $A\subset\nno$ be a nonempty set such that
$\Phi_s\mid A(X)$ for all $s\in S$. What is the minimal size of $A$?
In other words, what can we say about the quantity
\begin{equation}\label{capmin1}
\capmin(S) : = \min \{ \vert A \vert : \ A\neq\emptyset\hbox{ and }\Phi_s(X)\mid A(X)\hbox{ for all }s\in S\}?
\end{equation}

A natural question is whether the minimum in (\ref{capmin1}) is attained by sets $A$ that have a particularly simple ``fibered" structure on each scale $s\in S$. (We provide the definitions in Section \ref{subsec-fibered} below.) This turns out to be false in general, with one counterexample already given in \cite[Section 6.3]{LM}. In this article, we construct such counterexamples with the additional assumption that the least common multiple of the elements of $S$, which we denote by $\lcm(S)$, has only two distinct prime factors.
The latter assumption imposes strong structural constraints on any set $A$ contributing to (\ref{capmin1}) (see Lemma \ref{structure-thm}), so that examples with two prime factors are more unexpected and
significantly more difficult to construct. On the other hand, we are able to identify a number of special cases where the ``fibered lower bound" does hold.

Our interest in lower bounds on $\capmin(S)$ is motivated by several potential applications. One concerns the Coven-Meyerowitz conjecture on characterizing finite integer tiles; we discuss the conjecture and its relation to lower bounds on (\ref{capmin1}) in Section \ref{CM-subsection}. Similar issues have also arisen in the study of Fuglede's spectral set conjecture. In its original formulation \cite{Fug}, the conjecture states that a set $\Omega\subset \RR^n$ of positive $n$-dimensional Lebesgue measure tiles $\RR^n$ by translations if and only if the space $L^2(\Omega)$ admits an orthogonal basis of exponential functions (we refer to such sets as {\em spectral}).
While the original conjecture is now known to be false in its full generality, there remain important special cases where its status is unknown. In dimension 1, the Coven-Meyerowitz conjecture is known to imply that tiles are spectral sets (see \cite{dutkay-lai} for a summary of the argument). Conversely, to prove that spectral sets in a given group are tiles, one has to show that sets $A$ with small cardinality and many cyclotomic divisors of $A(X)$ must have very rigid structure. See e.g., \cite{KMSV, KMSV2} for examples and further references.

In a different direction, the question of bounding (\ref{capmin1}) from below has also come up in the study of the {\it Favard length} of product Cantor sets in the plane, an important question in geometric measure theory; see \cite{BLV,LM} for more details.

\subsection{The fibered lower bound}\label{subsec-fibered}
Let $N\in\NN$, and let $p$ be a prime such that $p\mid N$. We define
\begin{equation}\label{N-fiber}
F^N_p(X)=\Phi_p(X^{N/p})= 1+X^{N/p}+\dots+X^{(p-1)N/p}.
\end{equation}
This is a special case of a {\it $p$-fiber  on scale $N$} (see Section \ref{subsec-coord} for the definition).
We note that
$$
F^N_p(X)=\frac{X^N-1}{X^{N/p}-1},
$$
so that $\Phi_s \mid F^N_p$ if and only if $s \mid N$ but $s\nmid\frac{N}{p}$. In other words, if $\alpha$ is the exponent such that $p^\alpha\parallel N$, then $\Phi_s \mid F^N_p$ if and only if $p^\alpha  \mid s \mid N$. (Here and below, we use the notation $p^\alpha\parallel N$ to indicate that $p^\alpha  \mid N$ but $p^{\alpha+1}\nmid N$.)
In particular, the set $F^N_p=\{0,N/p,\dots,(p-1)N/p\}$ with the mask polynomial $F^N_p(X)$ has cardinality $p$, but $F^N_p(X)$ may have as many cyclotomic factors as we like, depending on $N$. Thus a large number of cyclotomic divisors of $A(X)$ does not, by itself, guarantee that $A$ has large cardinality. To force an increase in size, we need additional assumptions on $A$, $S$, or both. (This was already noted in \cite{LM}.)

More generally, we will say that a set $A\subset\nno$ is {\it fibered} on scale $N$ if there exists a prime $p \mid N$ such that
\begin{equation}\label{N-fibered}
F^N_p(X) \mid A(X).
\end{equation}
Of course, if $\Phi_s \mid F^N_p$ and (\ref{N-fibered}) holds, then $\Phi_s$ also divides $A(X)$.

For a finite and nonempty set $S\subset(\nn\setminus\{1\})$, we define $\capfib(S)$ to be the smallest size of a nonempty set $A\subset\nno$ such that $A$ is fibered on each scale $s\in S$.
For any such $A$, we clearly have $\Phi_s \mid A$ for each $s\in S$ (although $A$ may also have other cyclotomic divisors). Hence
\begin{equation}\label{min-fibered}
\capmin(S)\leq \capfib(S).
\end{equation}

An easy case when the equality holds in (\ref{min-fibered}) is as follows.

\begin{lemma}\label{lemma-prime-power}
Let $p$ be a prime number. Assume that $S=\{p^{\alpha_1},\dots,p^{\alpha_m}\}$, where $\alpha_1,\dots,\alpha_m\in\nn$ are all distinct. Then
\begin{equation}\label{eq-primepower}
\capmin(S)= \capfib(S)=p^m.
\end{equation}
\end{lemma}

\begin{proof}
Suppose that $\Phi_{p^{\alpha_1}}(X)\dots \Phi_{p^{\alpha_m}}(X)$ divides $A(X)$. Then
$$
p^m=\Phi_{p^{\alpha_1}}(1)\dots \Phi_{p^{\alpha_m}}(1)\ \Big| \ A(1)=|A|.
$$
In particular, $|A|\geq p^m$. Furthermore, let
$$
A_0:=\left\{ \sum_{j=1}^m a_jp^{\alpha_j-1}:\ a_j\in\{0,1,\dots,p-1\},\ j=1,\dots,m \right\}.
$$
Then $|A_0|=p^m$ (it is easy to see that the elements of $A_0$ listed above are all distinct), and by (\ref{e-primecyclo}), $A_0$ has the mask polynomial $A_0(X):=\Phi_{p^{\alpha_1}}(X)\dots \Phi_{p^{\alpha_m}}(X)$. Since
$$
\Phi_{p^\alpha}(X)=1+X^{p^{\alpha-1}}+\dots+X^{(p-1)p^{\alpha-1}}
=F^{p^\alpha}_p(X),
$$
$A_0$ is fibered on each of the scales $p^{\alpha_1},\dots,p^{\alpha_m}$.  This proves the lemma.
\end{proof}

The question is significantly more difficult when $\lcm(S)$ has two or more distinct prime factors.
In Section \ref{sec-fib-bound-works}, we investigate
several special cases when (\ref{min-fibered}) holds with equality. In particular, we prove the following.

\begin{theorem}\label{thm-3divisors}
Let $S\subset(\nn\setminus\{1\})$ satisfy $1\leq |S|\leq 3$, and assume that $\lcm(S)$ has at most two distinct prime factors. Then
$\capmin(S)= \capfib(S)$.
\end{theorem}

However, it is also possible for the inequality to be strict, and this can happen even if $\lcm(S)$ has only two prime factors.

\begin{theorem}\label{thm-smallexamples}
There exist finite and nonempty sets $S\subset(\nn\setminus\{1\})$ such that
$$
\capmin(S)< \capfib(S)
$$
and $\lcm(S)$ has two distinct prime factors.
\end{theorem}

\subsection{Integer tilings and the Coven-Meyerowitz conjecture}
\label{CM-subsection}

Let $A\subset\ZZ$ be finite and nonempty.
We say that $A$ {\em tiles the integers by translations} if
there exists a translation set $T\subset\ZZ$ such that every integer $n\in\ZZ$ can be written uniquely as $n=a+t$ with $a\in A$ and $t\in T$.

It is well known \cite{New} that any tiling of $\ZZ$ by a finite set $A$ must be periodic, so that there exists an $M\in\NN$ and a finite set $B\subset \ZZ$ such that $T=B\oplus M\ZZ$. 
In other words, $A\oplus B$ mod $M$ is a factorization of the cyclic group $\ZZ_M$. We write this as
\begin{equation}\label{mask-e0}
A\oplus B=\ZZ_M.
\end{equation}
By translational invariance, we may assume that
$A,B\subset \nno$.
Then (\ref{mask-e0}) can be rewritten in terms of the mask polynomials of $A$ and $B$:
\begin{equation}\label{mask-e1}
 A(X)B(X)\equiv 1+X+\dots+X^{M-1}\ \mod (X^M-1).
\end{equation}
Since $1+X+\dots+X^{M-1}=\prod_{s \mid M,s\neq 1}\Phi_s(X)$,
(\ref{mask-e0}) is further equivalent to
\begin{equation}\label{mask-e2}
  |A||B|=M\hbox{ and }\Phi_s(X) \mid  A(X)B(X)\hbox{ for all }s \mid M,\ s\neq 1.
\end{equation}
Since $\Phi_s$ are irreducible in $\mathbb{Q}[x]$, each $\Phi_s(X)$ with $s \mid M$ must divide at least one of $A(X)$ and $B(X)$.

Let $S_A^*$ be the set of all prime powers $p^\alpha$ such that $\Phi_{p^\alpha}(X)$ divides $A(X)$.
Consider the following conditions:

\smallskip
{ \it (T1) $|A|=A(1)=\prod_{s\in S_A^*}\Phi_s(1)$,}

\smallskip
{ \it (T2) if $s_1,\dots,s_k\in S_A^*$ are powers of distinct
primes, then $\Phi_{s_1\dots s_k}(X)$ divides $A(X)$.}

\noindent
Coven and Meyerowitz \cite{CM} proved the following theorem.

\begin{theorem}\label{thm-CM} \cite{CM} Let $A\subset\NN\cup\{0\}$ be a finite set. Then:
\begin{itemize}

\item if $A$ satisfies (T1), (T2), then $A$ tiles $\ZZ$;

\item  if $A$ tiles $\ZZ$ then (T1) holds;

\item if $A$ tiles $\ZZ$ and $|A|$ has at most two distinct prime factors,
then (T2) holds.
\end{itemize}
\end{theorem}

Any $A\subset\nno$, regardless of tiling properties, satisfies
$\prod_{s\in S_A^*}\Phi_s(1) \mid |A|$. The property (T1) follows then from an easy counting argument; the same argument also implies that if $A\oplus B=\ZZ_M$, then any prime power cyclotomic polynomial $\Phi_{p^\alpha}$ with $p^\alpha \mid M$ must divide exactly one of $A(X)$ and $B(X)$. See
Lemma \ref{prime-power-disjoint} in Section \ref{sec-lowerT2-basic}.

The second condition (T2) is much deeper and more difficult to prove.
The statement that (T2) must hold for all finite sets $A$ that tile the integers has become known in the literature as the {\it Coven-Meyerowitz conjecture}.
Beyond Theorem \ref{thm-CM}, the methods of \cite{CM} allow further mild extensions under additional assumptions on the tiling period $M$, see \cite[Corollary 6.2]{LL1}, \cite[Theorem 1.5]{M}, \cite[Proposition 4.1]{shi}, and the comments on \cite{Tao-blog}. More recently, further progress was made by {\L}aba and Londner \cite{LL1,LL2,LL-even, LL3}. The most general case where (T2) is currently known is given in \cite[Corollary 1.4]{LL3}.

One possible avenue of approach is to consider (T1) as an upper bound on the size of $A$, and ask whether a set obeying this bound may have additional cyclotomic divisors that would allow a failure of (T2) for its tiling complement. The details are as follows.

\begin{definition}\label{def-unsupported}
Let $A\subset\nn_0$, and let $\Phi_s(X)\mid A(X)$ for some $s\in\NN\setminus\{1\}$. We say that $\Phi_s$ is an {\em unsupported divisor of $A$} if:
\begin{itemize}
    \item[(i)] for every prime $p$ such that $p\mid s$, we have $p\mid |A|$,
\item[(ii)] for every prime power $p^\alpha$ such that $p^\alpha\parallel s$, we have $\Phi_{p^\alpha}\nmid A$.
\end{itemize}
\end{definition}

Let $M\in\NN$, and consider the following questions.

\begin{question}\label{Q1}
 If $A\subset\ZZ_M$ satisfies (T1), may it have unsupported divisors?
\end{question}

\begin{question}\label{Q2}
 If $A\subset\ZZ_M$ satisfies (T1) and (T2), may it have unsupported divisors?
\end{question}

\begin{question}\label{Q3}
Assume that $A\oplus B=\ZZ_M$. If (T2) holds for $A$, must $B$ also satisfy (T2)?
\end{question}

Trivially, the assumptions of Question \ref{Q2} are stronger than those of Question \ref{Q1}, hence a positive answer to the latter for some $M$ implies a positive answer to the former for the same $M$. Further, if the Coven-Meyerowitz conjecture is known to be true for some $M$ (in other words, both sets $A$ and $B$ in any tiling $A\oplus B=\ZZ_M$ must satisfy (T2)), this implies a positive answer to Question \ref{Q3} for the same $M$.

The next lemma states two less obvious relationships between the questions above and the Coven-Meyerowitz conjecture. To set the stage for it, we first note the following reduction from \cite[Lemma 2.5]{CM} (see also \cite[Lemma 6.2]{LL3}). Suppose that the Coven-Meyerowitz conjecture fails for some tiling $A\oplus B=\ZZ_M$, so that one of the sets $A$ and $B$ does not satisfy (T2). Then there exists a tiling $A'\oplus B'=\ZZ_{M'}$ for some $M'\mid M$ (obtained from the original tiling $A\oplus B=\ZZ_M$ via an explicit reduction procedure) such that (T2) also fails for one of the sets $A'$ and $B'$, and, additionally, each prime factor of $M'$ divides both $|A'|$ and $|B'|$.
We may therefore focus on tilings with this additional condition.

\begin{lemma}\label{Q123}
Assume that $A\oplus B=\ZZ_M$ for some $M\in\NN$, and that each prime factor of $M$ divides both $|A|$ and $|B|$.

\begin{itemize}

    \item[(i)] Suppose that the answer to Question \ref{Q1} is negative for this value of $M$. Then both sets $A$ and $B$ satisfy (T2).

\smallskip

\item[(ii)] Suppose that the answer to Question \ref{Q2} is negative for this value of $M$. Then, if (T2) holds for $A$, it also must hold for $B$.
\end{itemize}

\end{lemma}

\begin{proof}
Since $A\oplus B=\ZZ_M$, it follows from Theorem \ref{thm-CM} that both $A$ and $B$ satisfy (T1). Suppose that one of the sets, say $B$, does not satisfy (T2). Then there exists $s\geq 2$ such that $s \mid M$ and $\Phi_s\nmid B$, but $\Phi_{p^\alpha} \mid B$ for every prime power $p^\alpha\parallel s$. By (\ref{mask-e2}), we must have $\Phi_s \mid A$. Since (as pointed out above) no $\Phi_{p^\alpha}$ may divide both $A$ and $B$, $\Phi_s$ must be an unsupported divisor of $A$. This answers Question \ref{Q1} in the negative for that value of $M$, and it further answers Question \ref{Q2} in the negative for the same value of $M$ if we assume that $A$ satisfies (T2).
\end{proof}

It turns out that the answers to Questions \ref{Q1} and \ref{Q2}, without further constraints on $M$, are positive. Our examples are as follows.

\begin{theorem}\label{thm-T1-smallset}
There exist $M\in\NN$ and a nonempty set $A\subset\ZZ_M$ such that
$A$ satisfies (T1), $M$ has two distinct prime divisors, and
$A(X)$ has at least one unsupported cyclotomic divisor.
\end{theorem}

\begin{theorem}\label{thm-T2-smallset}
There exist $M\in\NN$ and a nonempty set $A\subset \ZZ_M$ such that $A$ satisfies both (T1) and (T2),  $M$ has four distinct prime divisors, and
$A(X)$ has at least one unsupported cyclotomic divisor.
\end{theorem}

However, the answers may be negative under additional assumptions on the tiling period or the number of scales. Our next theorem is an example of this.

\begin{theorem}\label{thm-T1T2-large}
Assume that a nonempty set $A\subset\ZZ_M$ satisfies (T1) and (T2), and that
$M$ has at most two distinct prime divisors. Then
$A(X)$ cannot have unsupported cyclotomic divisors.
\end{theorem}

At this time, our constructions do not provide directly any new information on the Coven-Meyerowitz conjecture. Theorem \ref{thm-CM} is already known when $|A|$ has at most two prime factors, and the set constructed in Theorem \ref{thm-T2-smallset} does not appear to have tiling complements that do not obey (T2). Nonetheless, Lemma \ref{Q123} implies that any counterexample to the Coven-Meyerowitz conjecture would have to involve a set that satisfies (T1) but has at least one unsupported divisor.
Our examples may be viewed as a partial step in that direction. We discuss this in more detail in Section \ref{subsec-4primes-discussion}.

We end this section with a comment on Definition \ref{def-unsupported}. If the condition (i) is dropped from that definition, then
examples providing a positive answer to Questions \ref{Q1} and \ref{Q2} are much easier to construct. For instance, the set $A$ in Example \ref{ex-3primes1scale} satisfies (T1); it also satisfies (T2) trivially, since $|A|$ is a prime number. However, the additional cyclotomic divisor $\Phi_{p_1p_2}$ in that example satisfies $(p_1p_2,|A|)=1$, and it is well known (see the paragraph before Lemma \ref{Q123}) that such divisors have no relevance to the Coven-Meyerowitz conjecture.

\subsection{Organization of the paper}\label{subsec-org}

In Section \ref{sec-cyclotools}, we transfer the problem to the setting of multisets in cyclic groups. We also review the basic cyclotomic divisibility tools available in the literature, such as array coordinates, grids, fibers, and cuboids.
The fibered lower bound is discussed in detail in Section \ref{sec-fibered}, where we also provide a way to evaluate it using the {\it assignment functions} defined there.

The next few sections are devoted to new methods developed for the purpose of this paper. We start with a truncation procedure (Section \ref{sec:structure}) that will allow us, in some cases, to simplify the question by removing ``unnecessary" scales. This is also where we define the exponent sets $\capexp_i(S)$ used throughout the rest of the article. In Section \ref{multiscale-cuboid}, we set up multiscale cuboid arguments, similar to those in \cite{LL1} and \cite{LM} (and based on them, to some extent) but adapted to our needs here.  Finally, in Section \ref{sec-long-fibers} we prove a multiscale generalization of the de Bruijn-R\'edei-Schoenberg structure theorem (Proposition \ref{cuboid}) in terms of the {\it long fibers} defined there.

In Section \ref{sec-fib-bound-works}, we identify several special cases when the equality $\capmin(S)=\capfib(S)$ holds. This includes the case when $|S|\leq 3$ and $\lcm(S)$ has two distinct prime divisors (Theorem \ref{thm-3divisors-ZM}, proving Theorem \ref{thm-3divisors} stated above), as well as certain cases when $\lcm(S)$ has more than two distinct prime factors but $S$ has a particularly simple structure.

Examples with $\capmin(S)<\capfib(S)$ are presented in Section \ref{sec-fib-bound-fails}. After presenting a simple example with 3 prime factors (Example \ref{ex-3primes1scale}), in Section \ref{sec-recombine-2primes} we move on to the more difficult examples
where $|A|$ has only 2 distinct prime factors. These examples prove Theorem \ref{thm-smallexamples}, and, since they all satisfy (T1), they
also prove Theorem \ref{thm-T1-smallset}.

Next, we address the more difficult Question \ref{Q2} from Section \ref{CM-subsection}. In Section \ref{sec-lowerT2-basic}, we prove
a structure result under the (T2) assumption. In Section \ref{sec-lowerT2-diagonal}, we identify an easy case when the answer to Question \ref{Q2} is negative. We then prove
Theorem \ref{thm-T1T2-large} in Section \ref{sec-lowerT2-2primes}.

Finally, the example
in Theorem \ref{thm-T2-4primes}
proves Theorem \ref{thm-T2-smallset}, with follow-up discussion in Section \ref{subsec-4primes-discussion}.


\section{Cyclotomic divisibility tools}\label{sec-cyclotools}

\subsection{Multisets}
It will be easier to work in a cyclic group setting. Suppose that we want to prove lower bounds on the size of sets $A\subset\nno$ such that $\Phi_s(X)\mid A(X)$ for all $s$ in a fixed, finite set $S\subset(\NN\setminus\{1\})$.
Let $M=\lcm(S)$, and consider $A\bmod M$ as a multiset in $\ZZ_M$ with the mask polynomial $A(X)\bmod (X^M-1)$. For any $s \mid M$, we have $\Phi_s\mid (X^M-1)$, so that $\Phi_s \mid A$ if and only if $\Phi_s\mid (A\bmod M)$. However, $A\bmod M$ need not be a set (since two or more elements of $A$ may be congruent mod $M$), hence we need to introduce notation for multisets in $\ZZ_M$.

We use $\calm(\ZZ_M)$ to denote the set of all multisets in $\ZZ_M$ with weights in $\ZZ$ (so that both positive and negative weights are allowed). For $a\in\ZZ_M$, we write $w_A(a)$ to denote the weight of $a$ in $A$. We also define the mask polynomial of the multiset $A$ by
\begin{equation}\label{eq:weightfunctiondefn}
A(X)=\sum_{a\in\ZZ_M} w_A(a) X^a.
\end{equation}
In particular, $A\in\calm(\ZZ_M)$ is a set if and only if $w_A(x)\in\{0,1\}$ for all $x\in\ZZ_M$. In that case, the above notation is consistent with the notation used in the introduction.

We use $A+B$ to denote the ``weighted union'' of multisets, so that $(A+B)(X)=A(X)+B(X)$
and $w_{A+ B}(x)=w_A(x)+w_B(x)$.
We use convolution notation for sumsets, with $(A*B)(X)=A(X)B(X)$. If $B=\{b\}$ is a singleton, we will write $b*A=B*A$.
The {\em support} of a multiset $A\in\calm(\ZZ_M)$ is the set $\{x\in\ZZ_M:\ w_A(x)\neq 0\}$. If $A\in\calm(\ZZ_M)$ (not necessarily with positive weights), and if $Y\subset\ZZ_M$ is a set, we will use $A\cap Y$ to denote the restriction of $A$ to $Y$. Thus $A\cap Y\in\calm(\ZZ_M)$, with weights
$$
w_{A\cap Y}(x)=w_{A}(x)w_Y(x).
$$
For $A\in\calm(\ZZ_M)$, we use $|A|$ to denote the ``total mass" of $A$, defined by
$$
|A|=\sum_{a\in\ZZ_M} w_A(a).
$$

We use $\calm^+(\ZZ_M)$ to denote the set of those multisets $A\in\calm(\ZZ_M)$ whose weights are all nonnegative and whose total mass $|A|$ is positive. (The latter requirement guarantees that $A$ is not the empty set.)
Abusing the notation slightly, we will write that two multisets $A,B\in\calm^+(\ZZ_M)$ satisfy $A\subset B$ if $w_A(x)\leq w_B(x)$ for all $x\in\ZZ_M$.

Our main results will be proved for multisets in cyclic groups, but it is easy to translate them back to the integer setting. In particular,
with the above notation, we have
\begin{equation}\label{eq-minS}
\capmin(S) =\min\{|A|:\ A\in\calm^+(\ZZ_M),\ \Phi_s\mid A\hbox{ for all }s\in S\},
\end{equation}
for any $M$ such that $\lcm (S)\mid M$. Indeed, if $A\subset\nno$ is a set such that $\Phi_s\mid A$ for all $s\in S$,
then the multiset $A':=(A\bmod M)$ in $\calm^+(\ZZ_M)$ satisfies $|A'|=|A|$ and
$\Phi_s\mid A'(X)$ for all $s\in S$. Conversely, let $A'\in\calm^+(\ZZ_M)$ be a multiset such that $\Phi_s\mid A'$ for all $s\in S$. We represent $\supp(A')\subset\ZZ_M$ as a subset of $\{0,1,\dots,M-1\}$, and let
$$
A:=\bigcup_{a\in\supp(A')} \{a,a+M,\dots,a+(w_A(a)-1)M\}.
$$
Then $A\subset\nno$ is a set with $|A'|=|A|$, and since $A(X)\equiv A'(X)$ modulo $X^M-1$, divisibility by all $\Phi_s$ with $s\in S$ is preserved.

\subsection{Coordinates, grids, fibers}\label{subsec-coord}
Let $M=\prod_{i=1}^K p_i^{n_i}$ be the prime number factorization of $M$, where $p_1,\dots,p_K$ are distinct primes and $n_1,\dots,n_K\in\nn$.
By the Chinese Remainder Theorem, we may represent $\ZZ_M$ as
$$
\ZZ_M=\bigoplus_{i=1}^K \ZZ_{p_i^{n_i}}\, ,
$$
which may be viewed geometrically as a $K$-dimensional lattice. We define an explicit coordinate system on $\ZZ_M$ as follows.
Let $M_i = M/p_i^{n_i} = \prod_{j\neq i} p_j^{n_j}$. Each $x\in \ZZ_M$ may then be written uniquely as
\begin{equation}\label{array-cd}
x=\sum_{i=1}^K x_i M_i,\ \ x_i\in \ZZ_{p_i^{n_i}}.
\end{equation}

We will often need to work on many scales $N\mid M$, each scale corresponding to a different cyclotomic divisor of $A(X)$. Given $N\mid M$, any multiset $A\in \calm(\ZZ_M)$ induces a multiset $(A\bmod N)\in\calm(\ZZ_N)$, with weights
$$
w^N_A(x)=\sum_{y\in\ZZ_M, y\equiv x\bmod N}w_A(y).
$$
To simplify the notation, we will continue to denote this multiset by $A$ (instead of $A\mod N$) whenever this does not cause confusion.
The mask polynomial of $A\bmod N$ is $(A\bmod N)(X)=A(X)\bmod (X^N-1)$. For any $s\mid N$ we have $\Phi_s(X)\mid X^N-1$, so that $\Phi_s(X)\mid A(X)$ if and only if $\Phi_s(X)\mid (A\bmod N)(X)$.

 For $p_j \mid N \mid M$, a {\it $p_j$-fiber on scale $N$} is a translate of any set $F^N_j$ such that
\begin{equation}\label{def-fiber}
 F^N_j\equiv \{0, N/p_j,2N/p_j,\dots,(p_j-1)N/p_j\} \mod N
\end{equation}
 with $p_j$ indicating the {\it direction} of the fiber.
Equivalently, (\ref{def-fiber}) may be written as
$$
F^N_j(X)\equiv 1+X^{N/p_j}+\dots+X^{(p_j-1)N/p_j}\mod X^N-1.
$$
 Note that our terminology is slightly different from the convention in \cite{LL1, LL2}. We also note a slight inconsistency with the notation in (\ref{N-fiber}); however, this should not cause problems, since it will always be clear from the context whether the subscript refers to the actual prime or to its index in the list $\{p_1,\dots,p_K\}$.

 We will say that a multiset $A\in\calm(\ZZ_M)$ is {\it fibered in the $p_j$ direction on scale $N$}, or $p_j$-fibered on scale $N$ for short, if there is a polynomial $Q(X)$ with nonnegative integer coefficients such that
 \begin{equation}\label{def-fibered}
A(X) \equiv Q(X) F^N_{j} (X) \mod X^N - 1.
\end{equation}

For $D \mid N \mid M$, a {\em $D$-grid} in $\ZZ_N$ is a set of the form
$$
\Lambda^N(x,D):= x*D\ZZ_N=\{x'\in\ZZ_N:\ D \mid (x-x')\}
$$
for some $x\in\ZZ_N$.  In other words, a $D$-grid is a coset of $D\ZZ_N\simeq \ZZ_{N/D}$ in $\ZZ_N$.

When $N=M$, we will omit the superscript $M$ to simplify the notation, so that $F_j=F_j^M$ and $\Lambda(x,D)=\Lambda^M(x,D)$.

\subsection{Cuboids}
We will use the following notation from \cite{LL1}. For multisets $\Delta\in\calm(\ZZ_N)$, where $N \mid M$, we define the \textit{$\Delta$-evaluations of $A\in\calm(\ZZ_M)$ in $\ZZ_N$:}
\begin{equation}\label{delta-eval}
\bbA^N[\Delta]=\sum_{x\in\ZZ_N}w_A^N(x)w_\Delta^N(x).
\end{equation}
The following special case is of particular interest.

\begin{definition}
\label{def-N-cuboids}
Let $M$ and $N$ be as above, and let $ \mathfrak{J}=\{j\in\{1,\dots,K\}:\ p_j \mid N\}$.
An \emph{ $N$-cuboid} is a multiset $\Delta \in \calm (\zz_N)$ associated with a mask polynomial of the form
\begin{equation}\label{def-delta}
\Delta(X)= X^c\prod_{j\in\mathfrak{J}} (1-X^{d_jN/p_j})
\end{equation}
 with $(d_j,p_j)=1$ for all $j\in\mathfrak{J}$.
\end{definition}

The geometric interpretation of $N$-cuboids, where $N=\prod_{j=1}^K p_j^{\alpha_j}$, is as follows. Let
\begin{equation}\label{eq-def-PN}
\mathcal{P}(N) : = \{p  :\ p  \mid N,\ p \hbox{ is prime} \},
\end{equation}
\begin{equation}\label{eq-def-DN}
D(N):= \frac{N}{\prod_{p\in\mathcal{P}(N)}p}= \prod_{j=1}^K p_j^{\gamma_j},\ \hbox{ where }
\gamma_j=\max(0,\alpha_j-1)\hbox{ for }j=1,\dots,K.
\end{equation}
(The denominator $\prod_{p\in\mathcal{P}(N)} p$ is also known as the {\it radical} of $N$.) Then the ``vertices'' of a cuboid $\Delta$,
$$
x_{\vec{\epsilon}} : = c + \sum\limits_{j \in \mathfrak{J}} \epsilon_j d_j\frac{N}{p_j} : \vec{\epsilon} \in \{0,1\}^{\vert \mathfrak{J} \vert},
$$
form a full-dimensional rectangular box in the grid $\Lambda^N(c,D(N))$, with one vertex at $c$ and alternating $\pm 1$ weights
$$w_{\Delta} (x_{\vec{\epsilon}}) = (-1)^{\sum\limits_{j \in \mathfrak{J}} \epsilon_j}.$$

The following cyclotomic divisibility test has been known and used previously in the literature. The equivalence between (i) and (iii) is the de Bruijn-R\'edei-Schoenberg theorem on the structure of vanishing sums of roots of unity (see \cite{deB, LamLeung, Mann, Re1, Re2, schoen}). For the equivalence (i) $\Leftrightarrow$ (ii), see e.g.  \cite[Section 3]{Steinberger}, \cite[Section 3]{KMSV}.

\begin{proposition}\label{cuboid}
Let $A\in\calm(\ZZ_M)$. Then the following are equivalent:
\begin{itemize}

\item[(i)] $\Phi_N(X) \mid A(X)$,

\item[(ii)] For all $N$-cuboids $\Delta$, we have
$\bbA^N[\Delta]=0$,

\item[(iii)] $A$ mod $N$ is a linear combination of $N$-fibers, so that
$$A(X)=\sum_{i:p_i \mid N} P_i(X)F^N_i(X) \mod X^N-1,$$
where $P_i(X)$ have integer (but not necessarily nonnegative) coefficients.
\end{itemize}

\end{proposition}

Proposition \ref{cuboid} can be strengthened if $N$ has only two distinct prime factors. This goes back to the work of de Bruijn \cite{deB}; a self-contained proof is provided in \cite[Theorem 3.3]{LamLeung}.

\begin{lemma}\label{structure-thm}
Let $A\in\calm^+(\ZZ_M)$.
Assume that $\Phi_{N} \mid A$, where $N$ has two distinct prime factors $p_1,p_2$. Then $A$ mod $N$ is a linear combination of $N$-fibers with nonnegative weights. In other words,
$$A(X)=P_1(X)F^N_1(X)+ P_2(X)F^N_2(X) \mod X^N-1,$$
where $P_1,P_2$ are polynomials with nonnegative coefficients.

\end{lemma}

It is well known that the positivity in Lemma \ref{structure-thm} does not hold when $N$ has 3 or more distinct prime factors. There are many examples of this in the literature, see e.g., the ``minimal relations" listed by Poonen and Rubinstein \cite[Table 1]{PR}, or the unfibered structures in \cite[Sections 5 and 6]{LL2} in the case when $N$ has 3 prime factors.

The following consequence of Proposition \ref{cuboid} will be used often.

\begin{lemma}\label{grid-split}
Assume that $N \mid M$ and $A\in\calm(\ZZ_M)$.
Then $\Phi_N \mid A$ if and only if $\Phi_N \mid (A\cap\Lambda)$ for every $D(N)$-grid $\Lambda$ in $\ZZ_M$.
\end{lemma}

\begin{proof}
This follows by replacing $A\in\calm(\ZZ_M)$ by $(A\bmod N)\in\calm(\ZZ_N)$ as described above, then applying
the equivalence (i) $\Leftrightarrow$ (ii) in Proposition \ref{cuboid}.
\end{proof}



\section{Lower bound for fibered sets}\label{sec-fibered}

Let $M = \prod_{k=1}^{K} p_k^{n_k}$,  where $p_1,\dots, p_K$ are distinct primes and $n_1,\dots,n_K \in\NN$.
For $s\in\NN$, we will use the notation
\begin{align*}
\mathcal{D}(s) &: = \{d \in \mathbb{N} : \ d \mid s,\ d\neq 1 \}.
\end{align*}

Let $S \subset \mathcal{D}(M)$ be non-empty, and let $\capmin(S)$ be given by (\ref{eq-minS}).
Recall also that in Section \ref{subsec-fibered} we defined $\capfib(S)$ to be the minimal size of a nonempty set $A\subset\nno$ such that $A$ is fibered in some direction on each scale $s\in S$.
By the same argument as in the proof of (\ref{eq-minS}), we may consider multisets $A\in \calm^+(\ZZ_M)$ instead of sets $A\subset\nno$.
We now indicate how to evaluate $\capfib(S)$.

\begin{definition}
Let $S \subset \mathcal{D}(M)$.
 An \textit{assignment function} is any function $\sigma : S \rightarrow \{1,\ldots,K\}$ such that $$\sigma (s) \in \{i : p_i \mid s \}.$$
Given $A \in \calm^+(\mathbb{Z}_M)$ and an assignment function $\sigma$, we say that $A$ is $(S , \sigma)$-\emph{fibered} if, for every $s \in S$, the associated multiset $A$ mod $s$ is fibered in the $p_{\sigma(s)}$ direction on the scale $s$.
\end{definition}

\begin{proposition}\label{prop-multifibered}
Let $S\subset\calD(M)$, and let $\sigma : S \rightarrow \{1,\ldots,K\}$ be an assignment function.  For each $i$,  let
$$
\capexp_i(S, \sigma) : = \{ \alpha\in\nn : \exists \, s \in S \textrm{ with } (s, p_i^{n_i}) = p_i^{\alpha} \textrm{ and } \sigma (s) = i \}.
$$
(We emphasize that the exponent 0 is not included above.)
Let $E_i(S,\sigma) : = \# \capexp_i(S, \sigma )$, and
$$
\capfib (S,\sigma) : = p_1^{E_1(S,\sigma)} \cdots p_K^{E_K(S,\sigma)}.
$$
In the special case when $\sigma(s)\equiv i$ for all $s\in S$, we will write $\capfib(S,\sigma)=\capfib(S,i)$.

Then
\begin{equation}\label{best-we-can-hope-for}
\capfib(S)= \min_{\sigma} \capfib (S,\sigma),
\end{equation}
with the minimum taken over all assignment functions $\sigma$. In particular, we have $\capmin(S)\leq \min_{\sigma} \capfib (S,\sigma)$.
\end{proposition}

\begin{proof}
We first prove that $\capfib(S) \geq \min_{\sigma} \capfib (S,\sigma)$. Indeed, suppose that $A \in \calm^+(\mathbb{Z}_M)$ is fibered on each scale $s\in S$. Then for each $s\in S$ there exists a prime $p_{i(s)} \mid s$ such that $A$ is fibered in the $p_{i(s)}$ direction on scale $s$. This defines an assignment function via $\sigma(s)=i(s)$ such that $A$ is $(S , \sigma)$-fibered. We fix this $\sigma$, and write $E_i:= E_i(S,\sigma)$ for short.

Next, we claim that
\begin{equation}\label{sigma-e21}
\hbox{ if }A\hbox{ is }(S,\sigma)\hbox{-fibered, then }\capfib(S,\sigma)\mid |A|.
\end{equation}
In particular, $
|A|\geq\capfib(S,\sigma)$ as required. To prove (\ref{sigma-e21}), it suffices to prove that $p_i^{E_i}\mid |A|$ for each $i\in\{1,\dots,K\}$. Fix such $i$, assume that
$$
\capexp_i(S,\sigma)=\{\alpha_1,\dots,\alpha_{E_i}\}
$$
(if this is an empty set, then $E_i=0$ and there is nothing to prove),
and let $s_1,\dots,s_{E_i}$ be elements of $S$ such that $\sigma(s_j)=i$ and $p_i^{\alpha_j}\parallel s_j$.
Since $A$ is $p_i$-fibered on each scale $s_j$, we have
$$
F_i^{s_j}(X)\mid A(X) \mod (X^{s_j}-1),
$$
where $F^{s_j}_i(X)=(X^{s_j}-1)/(X^{s_j/p_i}-1)$. Since $p_i^{\alpha_j}$ divides $s_j$ but not $s_j/p_i$, it follows that $\Phi_{p_i^{\alpha_j}} \mid A$. Therefore
$$
p_i^{E_i}=\Phi_{p_i^{\alpha_1}}(1)\dots \Phi_{p_i^{\alpha_{E_i}}}(1)\mid A(1)=|A|,
$$
proving (\ref{sigma-e21}).


For the converse inequality, given an assignment function $\sigma$, we
give an explicit ``standard" $(S,\sigma)$-fibered set $A^\flat=A_{S,\sigma}^\flat$ such that $|A^\flat|=\capfib(S,\sigma)$ and $\Phi_s \mid A$ for all $s\in S$. The construction follows \cite{CM} and was also used in \cite{LL1,LL2} in the context of integer tilings.
Let $M_i:=M/p_i^{n_i}$, and define
\begin{equation}\label{standard-set}
\begin{split}
A^\flat(X) &= \prod_{i=1}^{K} \left[ \prod_{\beta\in \textsf{EXP}_{i} (S, \sigma)}  \Phi_{p_i}\left(X^{M_ip_i^{\beta-1}}
\right)\right]
\\
&=  \prod_{i=1}^{K} \left[ \prod_{\beta\in\textsf{EXP}_{i} (S, \sigma)}
\left(1+ X^{M_ip_i^{\beta-1}} + \dots + X^{(p_i-1)M_ip_i^{\beta-1}} \right)\right] .
\end{split}
\end{equation}
Since $\Phi_p(1)=p$ for prime $p$, we have
$$
|A^\flat|=A^\flat(1)= \prod_{i=1}^{K} \left[ \prod_{\beta\in\textsf{EXP}_{i} (S, \sigma)}  p_i \right] = \capfib (S,\sigma).
$$
Next, let $s\in S$, and let $i=\sigma(s)$ so that $p_i^{\beta}\parallel s$ for some $\beta\in \textsf{EXP}_{i} (S, \sigma)$.
Observe that
\begin{equation}\label{cyclo-divisors}
\Phi_{p_i}\big(X^{M_ip_i^{\beta-1}}\big)= \frac{1-X^{M_ip_i^{\beta}}}{1-X^{M_ip_i^{\beta-1}}}
= \prod_{u \mid M_ip_i^{\beta}, u\nmid M_ip_i^{\beta-1}}\Phi_u(X),
\end{equation}
\begin{equation}\label{cyclo-divisors2}
F^s_i(X)= \frac{1-X^{s}}{1-X^{s/p_i}}
= \prod_{u \mid s, u\nmid\, (s/p_i)}\Phi_u(X).
\end{equation}
Hence, we have the chain of divisibility
$$
\Phi_s(X) ~ \Big| ~F^s_i(X)\ \Big| \ \Phi_{p_i}\left(X^{M_ip_i^{\beta-1}}\right) \ \Big| \ A^\flat(X),
$$
proving both $\Phi_s\mid A^\flat$ and the fibering claim.
\end{proof}


\section{A truncation procedure}\label{sec:structure}

We introduce a truncation procedure that will allow us to reduce proving upper or lower bounds on $\capmin(S)$ to proving similar bounds with $S$ replaced by a simpler set.
Let $M = \prod_{k=1}^{K} p_k^{n_k}$,  where $p_1,\dots, p_K$ are distinct primes and $n_1,\dots,n_K \in\NN$.
We continue to use the coordinate representation
$$
\mathbb{Z}_M \ni x \equiv x_1 M_1 + \cdots + x_K M_K \mod M,
$$
where $M_i = M/p_i^{n_i}$ and $x_i \in \mathbb{Z}_{p_i^{n_i}}$.  We will also need the  $p_i$-adic expansion of $x_i$:
\begin{equation}\label{padic-digits}
x_i \equiv x_{i,0} + x_{i,1} p_i + \cdots + x_{i,n_i - 1} p_i^{n_i - 1} \mod p_i^{n_i}, \ x_{i,j}\in\{0,1,\dots,p_i-1\},
\end{equation}

For $S\subset\cald(M)$ and $1 \leq i \leq K$, we define
\begin{equation}\label{eq-defexp}
\textsf{EXP}_i(S) := \{\alpha \geq 1 :\  \exists \, s \in S \textrm{ with } p_i^{\alpha} \mid \mid s \},
\ \
E_i := \# \textsf{EXP}_i(S).
\end{equation}
It will be useful to arrange the sets $\textsf{EXP}_i(S)$ in increasing order:
\begin{equation}\label{e-trunk5}
\textsf{EXP}_i(S) : = \{\alpha_{i,1}, \cdots, \alpha_{i,E_i}\},\ \ 1\leq \alpha_{i,1}<\dots<\alpha_{i,E_i}.
\end{equation}
Let us use the convention $\alpha_{i,0}=0$.
We then have the following proposition.

\begin{proposition}\label{prop:truncation} {\bf (Truncations)}
Let $S\subset\cald(M)$, and let $A \in \calm (\mathbb{Z}_M)$ satisfy $\Phi_s \mid A$ for all $s\in S$.
Define $M' : = p_1^{E_1} \cdots p_K^{E_K}$. Then, there exists a multiset $A' \in \calm (\mathbb{Z}_{M'})$ satisfying
\begin{itemize}
\item[(i)] $A'(1) = A(1)$,

\medskip

\item[(ii)] For every $N =  p_1^{\alpha_{1,\ell_1}} \cdots p_K^{\alpha_{K,\ell_K}}$ with $\Phi_N(x) \mid A(x)$,  we have $\Phi_{N'} (X) \mid A'(X)$, where $N' : = p_1^{\ell_1} \cdots p_K^{\ell_K} \mid M'$.
\end{itemize}
Furthermore, if $A\in\calm^+(\ZZ_M)$, then $A'\in\calm^+(\ZZ_{M'})$.
\end{proposition}

Proposition \ref{prop:truncation} allows us to assume that $\textsf{EXP}_i(S) = \{1,2,\ldots,E_i\}$ for every $i \in \{1,\ldots,K\}$, so that there are no gaps in our set of exponents. We refer to the multiset $A' \in \calm (\mathbb{Z}_{M'})$ in Proposition \ref{prop:truncation} as the \textit{truncation of} $A$ \textit{relative to} $S$.

\begin{example}
\rm{
Suppose that $\Phi_s \mid A$ for all $s\in S$, where
$$S : = \{p_2^2,  p_1^3, p_2^4, p_1^3p_2^4, p_1^{10}, p_2^{10}, p_1^{10} p_2^{10}\}.$$
Then,  $\textsf{EXP}_1(S)= \{3, 10\}$ and $\textsf{EXP}_2(S) = \{ 2,4,10\}$ so that $M' := p_1^{E_1} p_2^{E_2} = p_1^2 p_2^3$. Proposition \ref{prop:truncation} then furnishes a multiset $A' \in \calm (\mathbb{Z}_{p_1^2 p_2^3})$ such that $A'(1) = A(1)$ and $\Phi_s \mid A'$ for all $s\in S':=  \{p_1, p_2, p_1^2, p_2^2, p_1 p_2^2, p_2^3, p_1^2 p_2^3\}$. The exponent sets associated to $A'$ are $\{1,2\}$ and $\{1,2,3\}$, with no gaps.
}
\end{example}

\begin{figure*}[h!]
    \centering
    \begin{subfigure}[t]{0.5\textwidth}
        \centering
        \includegraphics[width=.7\textwidth]{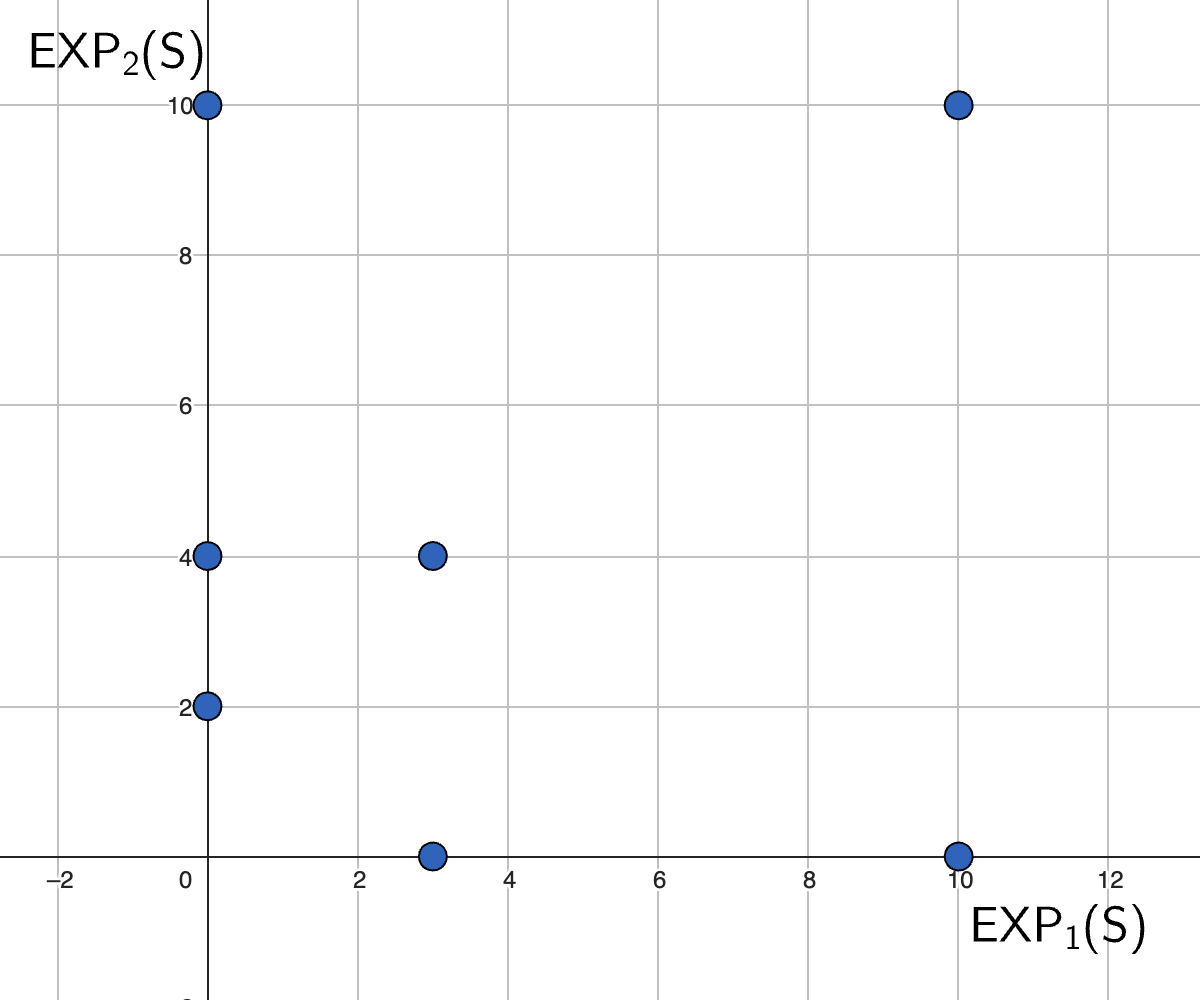}
    \end{subfigure}%
    ~
    \begin{subfigure}[t]{0.5\textwidth}
        \centering
        \includegraphics[width=.7\textwidth]{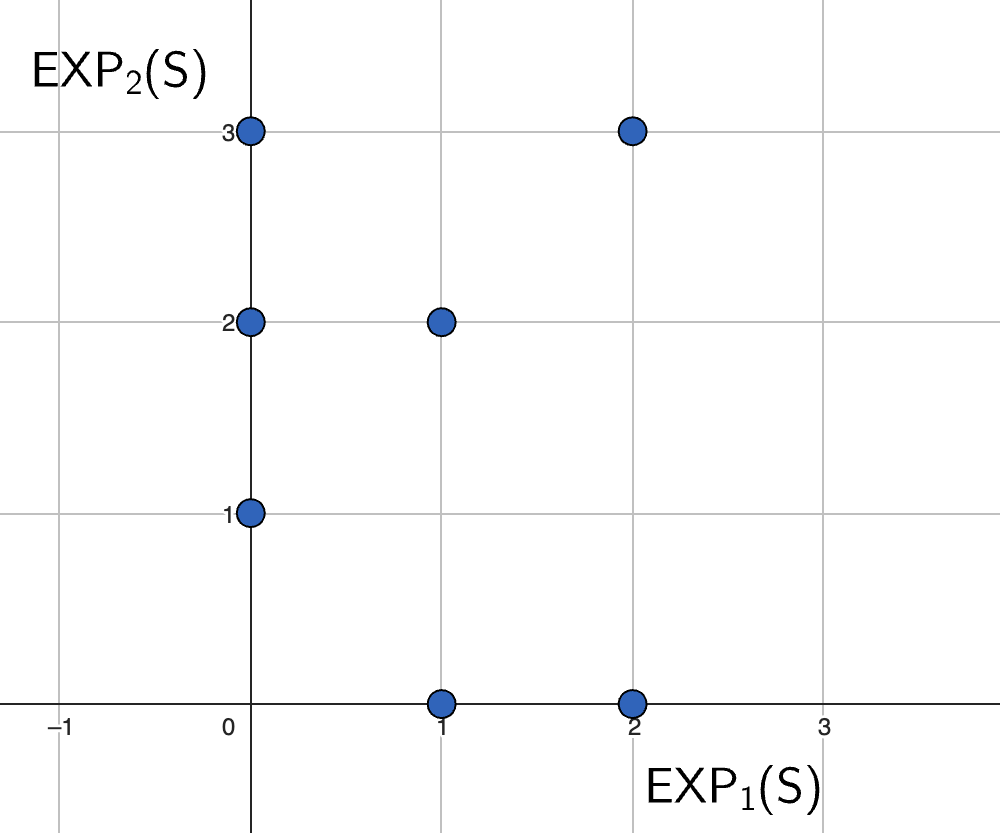}
	 \end{subfigure}
    \caption{The cyclotomic divisors of $A$ and the cyclotomic divisors of $A'$ .}
\end{figure*}

We now begin the proof of Proposition \ref{prop:truncation}.
We first define a family of mappings on $\mathbb{Z}_M$ which preserve cyclotomic divisibility on the scales we need, but remove ``unnecessary" $p_i$-adic digits.

\begin{definition}
Let $i \in \{1,\ldots,K\}$ and $1 \leq \alpha \leq n_i$ be given. Recalling the $p_i$-adic expansion of  $x_i\in \mathbb{Z}_{p_i^{n_i}} $ given in (\ref{padic-digits}), we
define a mapping $T_i^{\alpha} : \mathbb{Z}_{p_i^{n_i}} \rightarrow \mathbb{Z}_{p_i^{n_i}}$ by writing
$$
T_i^{\alpha} (x_i) : = x_i - x_{i,\alpha - 1} p_i^{\alpha - 1}, \quad \forall x_i \in \mathbb{Z}_{p_i^{n_i}}.
$$
so that $T_i^{\alpha}$ sends the $\alpha$-scale coordinate of $x_i$ to zero.  We further define a mapping $\mathbf{T}_i^{\alpha} : \mathbb{Z}_M \rightarrow \mathbb{Z}_M$ by writing in array coordinates,
$$
\mathbf{T}_i^{\alpha} (x_1,\ldots,x_K) : = \big(x_1,\ldots,T_i^{\alpha} (x_i),\ldots,x_K \big), \quad \forall (x_1,\ldots,x_K) \in \mathbb{Z}_{p_1^{n_1}} \times \cdots \times \mathbb{Z}_{p_K^{n_K}}.
$$
\end{definition}

We note that $\bfT_i^\alpha$ has the following property:
\begin{equation}\label{e-trunk0}
\forall y,z\in\ZZ_M,\ \  y_{j,\beta}=z_{j,\beta} \ \Rightarrow (\bfT_i^\alpha(y))_{j,\beta} = (\bfT_i^\alpha(z))_{j,\beta}
\end{equation}
for all $j\in\{1,\dots,K\}$ and $\beta\in\{0,1,\dots,n_j-1\}$. In other words, if some of the multiscale digits of $y,z$ are equal, then the corresponding digits of their images under $\bfT_i^\alpha$ are also equal. The converse implication fails when $j=i$ and $\beta=\alpha-1$, since then the corresponding (possibly non-equal) coordinates of $y$ and $z$ are both sent to $0$.

\begin{lemma}\label{lma:truncationfiber}
Let $N \mid M$ with $p_i^{\beta} \parallel N$ for some $0 \leq \beta \leq n_i$ and let $1 \leq \alpha \leq n_i$ be given.  Let $F \in \calm^+ (\mathbb{Z}_M)$ satisfy $F \equiv x * F_j^{N}$ mod $N$ for some $p_j \mid N$ and $x \in \mathbb{Z}_N$. Assume that at least one of the following holds:
\begin{itemize}
\item[(i)] $j\neq i$,
\medskip
\item[(ii)] $\beta \neq \alpha$.
\end{itemize}
Then $\mathbf{T}_i^{\alpha} (F) \equiv \mathbf{T}_i^{\alpha} (x) * F_j^N$ mod $N$.
\end{lemma}

\begin{proof}
Let $N$ and $F$ be as in the statement of the lemma. This means that $F$ is a set of $p_j$ elements such that if $y,z\in F$ are distinct, then $(y-z,N)=N/p_j$. It suffices to prove that for any such $y,z$ we also have
$$
(\bfT_i^\alpha(y)- \bfT_i^\alpha(z),N)=N/p_j.
$$
Let $y,z\in F$ be distinct, and let $N=\prod p_\ell^{\beta_\ell}$ be the prime factorization of $N$, so that $\beta_i=\beta$. We need to prove that
\begin{equation}\label{e-trunk1}
\left(\bfT_i^\alpha(y)- \bfT_i^\alpha(z),p_\ell^{\beta_\ell}\right)=
\begin{cases}
p_\ell^{\beta_\ell}& \hbox{ if }\ell\neq j,\\
p_j^{\beta_j-1}& \hbox{ if }\ell= j.
\end{cases}
\end{equation}
We have
$$
\bfT_i^\alpha(y)- \bfT_i^\alpha(z)= (y-z) - (y_{i,\alpha-1}-z_{i,\alpha-1})p_i^{\alpha-1} M_i.
$$
We consider three cases.
\begin{itemize}
\item If $\ell\neq i$, then $p_\ell^{\beta_\ell}\mid M_i$, so that $(\bfT_i^\alpha(y)- \bfT_i^\alpha(z),p_\ell^{\beta_\ell})= (y-z, p_\ell^{\beta_\ell})$ and (\ref{e-trunk1}) follows.
\medskip

\item Suppose $\ell =i$ but $i\neq j$. Then $p_i^{\beta}\mid y-z$, so that $y_{i,\gamma}=z_{i,\gamma}$ for $\gamma\leq\beta-1$.  By (\ref{e-trunk0}), the same holds for the corresponding digits of $\bfT_i^\alpha(y)$ and
$\bfT_i^\alpha(z)$, implying (\ref{e-trunk1}).

\medskip

\item Finally, assume that $\ell=i=j$. Then $p_i^{\beta-1}\parallel (y-z)$, so that $y_{i,\beta-1}\neq z_{i,\beta-1}$
and $y_{i,\gamma}=z_{i,\gamma}$ for $\gamma< \beta-1$. By (\ref{e-trunk0}), we have $(\bfT_i^\alpha(y))_{i,\gamma}=(\bfT_i^\alpha(z))_{i,\gamma}$ for $\gamma< \beta-1$. Furthermore, since $\beta\neq\alpha$ in this case, we have
$$
(\bfT_i^\alpha(y))_{i,\beta-1 }= y_{i,\beta-1}\neq z_{i,\beta-1}= (\bfT_i^\alpha(z))_{i,\beta-1}.
$$
Hence $ p_i^{\beta-1}\parallel (\bfT_i^\alpha(y)- \bfT_i^\alpha(z))$, and (\ref{e-trunk1}) holds again.
\end{itemize}
\end{proof}

\begin{corollary}\label{cor-trunk1}
Let $S\subset\cald(M)$, and let $A\in\calm(\ZZ_M)$ satisfy $\Phi_s \mid A$ for all $s\in S$. Assume that $\alpha\not\in \capexp_i(S)$. Then $\Phi_s\mid \bfT_i^\alpha(A)$ for all $s\in S$.
\end{corollary}

\begin{proof}
Let $S$ and $A$ satisfy the assumptions of the corollary. Fix $N\in S$. By the de Bruijn-R{\'e}dei-Schoenberg theorem (the equivalence of (i) and (iii) in Proposition \ref{cuboid}), we may write
$$
A(X)\equiv  \sum_{j:p_j \mid N} Q_j(X) F^N_j(X) \mod (X^N-1),
$$
where $Q_j$ are polynomials with integer coefficients. Since $\alpha\not\in \capexp_i(S)$, the assumptions of
Lemma \ref{lma:truncationfiber} are satisfied, hence
$\mathbf{T}_i^{\alpha}$ maps each fiber $x*F_j^N$ mod $N$ to a fiber $\mathbf{T}_i^{\alpha} (x) * F_j^N$ mod $N$. It follows that
$\mathbf{T}_i^{\alpha} (A)$ can also be written as a linear combination of fibers on scale $N$, and another application of the equivalence (i)$\Leftrightarrow$(iii) in Proposition \ref{cuboid} proves that $\Phi_N\mid \bfT_i^\alpha(A)$.
\end{proof}

\begin{corollary}\label{cor-trunk2}
For $S\subset\cald(M)$, we define a mapping $\mathbb{T}_S : \mathbb{Z}_M \rightarrow \mathbb{Z}_M$ via
\begin{equation}\label{eq:localizationmapping}
\mathbb{T}_S (x_1,\ldots,x_K) : = \bigg(\sum\limits_{\alpha_1 \in \textsf{EXP}_1(S)} x_{1,\alpha_1 -1} p_1^{\alpha_1 -1 }, \cdots, \sum\limits_{\alpha_K \in \textsf{EXP}_K(S)} x_{K, \alpha_K - 1} p_K^{\alpha_K - 1} \bigg).
\end{equation}
Assume that $A \in \calm (\mathbb{Z}_M)$ satisfies $\Phi_s \mid A$ for all $s\in S$. Then the multiset $\bbT_S(A) \in \calm (\mathbb{Z}_M)$, with weight function
\begin{equation}\label{truncated-weights}
w_{\bbT_S(A)} (x) : = \sum\limits_{\{y:\   \bbT_S(y) =x\}} w_A (y) \quad \forall x \in \mathbb{Z}_M,
\end{equation}
satisfies $\Phi_s\mid \bbT_S(A)$ for all $s\in S$. Furthermore, if $A\in\calm^+(\ZZ_M)$, then $\bbT_S(A)\in\calm^+(\ZZ_M)$.
\end{corollary}

\begin{proof}
This follows by observing that $\bbT_S$ is the composition of the mappings $\bfT_i^\alpha$, where $(i,\alpha)$ runs over all pairs such that $i\in\{1,\dots,K\}$ and $\alpha\in \{1,\dots,n_i\}\setminus \capexp_i(S)$, and applying Corollary \ref{cor-trunk1} iteratively to each such mapping. The last statement is a consequence of (\ref{truncated-weights}).
\end{proof}

\begin{proof}[Proof of Proposition \ref{prop:truncation}]
Let $S$, $A$, and $M'$ be as in the statement of the proposition. We enumerate the elements of each set $\capexp_i(S)$ in increasing order as in (\ref{e-trunk5}). We also equip $\ZZ_{M'}$ with a standard coordinate system similar to that in $\ZZ_M$.

We first let $\tilde A:=\bbT_S(A)\in\calm(\ZZ_M)$, then $|\tilde A|=|A|$ and $\Phi_s\mid A$ for all $s\in S$ by Corollary \ref{cor-trunk2}. We further have $\supp(\tilde A)\subset Y$, where
$$
Y:=\bbT_S(\ZZ_M) = \{y\in\ZZ_M:\ y_{i,\alpha}=0\hbox{ for all }(i,\alpha)\hbox{ such that }\alpha\not\in\capexp_i(S)\}.
$$
We define a bijection $\mathbb{U} : Y\to  \mathbb{Z}_{M'}$ by
$$
\mathbb{U} \bigg(\sum\limits_{j=0}^{E_1-1} x_{1,j} p_1^{\alpha_{1,j} -1 }, \cdots, \sum\limits_{j=0}^{E_K-1} x_{K, j} p_K^{\alpha_{K,j} - 1} \bigg)
=\bigg(\sum\limits_{j=0}^{E_1 - 1} x_{1,j} p_1^{j}, \cdots, \sum\limits_{j=0}^{E_K - 1} x_{K,j} p_K^{j} \bigg).
$$
Let $A' \in \calm (\mathbb{Z}_{M'})$ be the multiset defined via the weight equality
\begin{equation}\label{eq:truncationweigthdefn}
w_{A'}^{M'} (x) : = w_{\tilde{A}} (\mathbb{U}^{-1} (x) ), \quad \forall x \in \mathbb{Z}_{M'}.
\end{equation}
We clearly have $|A'|=|\tilde A|=|A|$. It remains to prove that
$$
\hbox{if }N =  p_1^{\alpha_{1,\ell_1}} \cdots p_K^{\alpha_{K,\ell_K}}\in S, \hbox{ then }
\Phi_{N'} (X) \mid A'(X), \hbox{ where }N' : = p_1^{\ell_1} \cdots p_K^{\ell_K}.
$$
Let $N\in S$, then $\Phi_N \mid \tilde A$ as noted above.
By the equivalence (i)$\Leftrightarrow$(iii)
in Proposition \ref{cuboid}, $\tilde A$ may be written as a linear combination of fibers on scale $N$. However,
$\mathbb{U}$ maps fibers on scale $N$ in $\mathbb{Z}_M$ to fibers on scale $N'$ in $\mathbb{Z}_{M'}$, so that $A'$ is a linear combination of such fibers. By another application of Proposition \ref{cuboid}, we have $\Phi_{N'}\mid A'$ as claimed.
Finally, if $A \in \calm^+ (\mathbb{Z}_M)$, then $A' \in \calm^+ (\mathbb{Z}_{M'})$
by \eqref{eq:truncationweigthdefn} and the last part of Corollary \ref{cor-trunk2}.
\end{proof}


\section{A Multiscale Cuboid Argument}\label{multiscale-cuboid}


Let $N \mid M$. Recall that
we defined $D(N)$ and $\mathcal{P}(N)$ in (\ref{eq-def-PN}) and (\ref{eq-def-DN}).
For $y \in \mathbb{Z}_N$, we continue to write
\begin{equation}\label{eq:gridforlemma}
\Lambda (y, D(N)) : = \{x \in \mathbb{Z}_N : D(N) \mid (x - y) \}.
\end{equation}

Let $p \in \mathcal{P}(N)$ with $p^{\alpha} \mid \mid N$. For each $\nu\in\{0,1,\dots,p-1\}$,  let $y_{\nu} = y + \nu N /p$. Then
$$
\Lambda(y,D(N)) = \bigcup_{\nu=0}^{p-1} \Lambda(y_\nu, pD(N)),
$$
which corresponds to a decomposition of the original grid $\Lambda(y,D(N))$ into those parts which are contained in the planes $\Pi (y_{\nu}, p^{\alpha}) : = \{x \in \mathbb{Z}_N : p^{\alpha} \mid (x - y_{\nu})\}$.

%

\begin{proposition}\label{lma:reductionindimension} Let $ A\in  \calm^+(\ZZ_N)$.
Suppose that $\Phi_N \mid A$ and that $p \in \mathcal{P}(N)$ with $p^{\alpha} \mid \mid N$.  Then at least one of the following holds.

\begin{enumerate}
\item $\Phi_{N} \Phi_{N/p} \cdots \Phi_{N/p^{\alpha}} \mid A$.

\medskip

\item For $\nu\in\{0,1,\dots,p-1\}$ and $a\in A$, define the multisets $A_{\nu,a} \subset A$ by
\begin{equation}\label{eq:Aplanesandgrids}
A_{\nu,a}:= A\cap \Lambda(y_\nu,pD(N))
\end{equation}
where $y_{\nu} = a + \nu N /p$.
Then there exists some $a \in A$ such that $A_{\nu,a}$ are nonempty for all $\nu\in\{0,1,\dots,p-1\}$.

\end{enumerate}
\end{proposition}

The proof of Proposition \ref{lma:reductionindimension} is based on multiscale cuboid argument similar to that in \cite[Section 5]{LL1}.
To simplify notation, we fix $N \mid M$, let $\mathfrak{J} = \{j : p_j \mid N \}$, and fix some specific prime $p = p_{j_0} \mid N$. We also let $\mathfrak{J}' : = \mathfrak{J} \setminus \{j_0\}$.
A \textbf{flat cuboid} is then a multiset $\Delta^{p} \in \calm (\mathbb{Z}_N)$ of the form
$$
\Delta^{p} (X) =  X^{c}  \prod_{j\,  \in \, \mathfrak{J}'} \big(1-X^{d_jN/p_j}\big) \mod X^{N} - 1,
$$
where $c,d_j \in \mathbb{Z}_N$ and $(d_j,p_j) = 1$ for every $j \in \mathfrak{J}'$.  If $N$-cuboids correspond to $\vert \mathfrak{J} \vert$-dimensional rectangular boxes,  then flat cuboids correspond to rectangular boxes of dimension $\vert \mathfrak{J} \vert$ - 1 contained in planes perpendicular to the $p$ direction.

Flat cuboids are useful in so far as the associated $\Delta^{p}$-evaluations of $A$ in $\mathbb{Z}_N$ indicate simultaneous divisibility by multiple cyclotomic polynomials.  Lemma \ref{lma:chaindivflat} is taken from \cite{LL1}, but a similar result (in a somewhat different language) also appears in \cite[Lemma 2.13]{KMSV2}.

\begin{lemma}\label{lma:chaindivflat} \cite[Example 5.9(2)]{LL1} Let $A\in\calm^+(\ZZ_N)$.
Let $p^{\alpha} \mid \mid N$, and assume that
$
\mathbb{A}^N [\Delta^{p}] = 0
$
for every flat cuboid $\Delta^{p}$ as defined above. Then
$$
\Phi_{N}\Phi_{N/p} \cdots \Phi_{N/p^{\alpha}} \mid A.
$$
\end{lemma}

Observe that any $N$-cuboid $\Delta$ as in (\ref{def-delta}) can be written as $\Delta (X) = \Delta_-^{p} - \Delta_+^{p}$ where
$$
\Delta_-^{p} (X) = X^{c}  \prod_{j\,  \in \, \mathfrak{J}'} \big(1-X^{d_jN/p_j}\big) \mod X^{N} - 1,
$$
$$
\Delta_+^{p}(X) = X^{c + dN/p}  \prod_{j\,  \in \, \mathfrak{J}'} \big(1-X^{d_jN/p_j}\big) \mod X^{N} - 1.
$$
Together with Proposition \ref{cuboid}(ii), this gives the following result for flat cuboids.

\begin{lemma}\label{lma:cuboidstackflat}
Let $y \in \mathbb{Z}_N$. For $\nu = 0,\ldots,p-1$, let $y_{\nu} : = y + \nu N/p$ and
$$
\Delta_{\nu}^{p} (X) : = X^{y_{\nu}} \prod_{j\,  \in \, \mathfrak{J}'} \big(1-X^{d_jN/p_j}\big) \mod X^{N} - 1.
$$
If $\Phi_N \mid A$, then
\begin{equation}\label{eq:flatequal}
\mathbb{A}^N [\Delta_{\nu}^{p}] = \mathbb{A}^N [\Delta_{\nu'}^{p}]
\end{equation}
for every $0 \leq \nu, \nu' \leq p - 1$.
\end{lemma}

\begin{proof}
For each $\nu = 0,\ldots,p - 2$, the multiset $\Delta \in \calm (\mathbb{Z}_N)$ with the mask polynomial
$
\Delta (X) = \Delta_{\nu}^{p} (X) - \Delta_{\nu + 1}^{p} (X)
$
is an $N$-cuboid. By Proposition \ref{cuboid}(ii), we have
$$
\mathbb{A}^N [\Delta_{\nu}^{p}] - \mathbb{A}^N [\Delta_{\nu + 1}^{p}] = \mathbb{A}^N [\Delta] = 0,
$$
for all $\nu = 0,\ldots,p - 2$. This proves the lemma.
\end{proof}

\begin{proof}[Proof of Proposition \ref{lma:reductionindimension}]
Let $p \in \mathcal{P} (N)$ with $p^{\alpha} \parallel N$.
Assume that $\Phi_N \mid A$, but
\begin{equation}\label{eq:ndivchain}
\Phi_N \cdots \Phi_{N/p^{\alpha}} \nmid A,
\end{equation}
By the contrapostive of Lemma \ref{lma:chaindivflat}, there is a flat cuboid
$$
\Delta_0^{p} (X) = X^{y} \prod_{j \in \mathfrak{J}'} (1 - X^{d_jN/p_j}) \mod X^{N} - 1,
$$
where $\mathfrak{J} := \{j : p_j \mid N \}$ and $(d_j,p_j) = 1$,
such that $\mathbb{A}^N [\Delta_0^{p}] \neq 0$.
In particular, this implies that
$\supp \, A \cap \supp \, \Delta_0^{p} \neq \emptyset,
$
so that without loss of generality we may assume that $y = a \in A$.

Let $y_\nu:= a+\nu N/p$ and
$$
\Delta_{\nu}^{p} (X) : = X^{y_{\nu}} \prod_{j\,  \in \, \mathfrak{J}'} \big(1-X^{d_jN/p_j}\big) \mod X^{N} - 1.
$$
By Lemma \ref{lma:cuboidstackflat}, we have $\mathbb{A}^N [\Delta_{\nu}^{p}] = c$ for each $\nu$ and some constant $c \neq 0$. In particular,
$\supp \, \Delta_{\nu}^{p} \cap \supp \, A \neq \emptyset$ for each $\nu$. Since
$\supp \, \Delta_{\nu}^{p} \subset  \Lambda(y_\nu,pD(N))  \textrm{ for each } \nu = 0,\ldots,p-1,$
this completes the proof.
\end{proof}

\begin{corollary} \label{cor-split}
Assume that $\Phi_N\mid A$. Then there exist a prime $p\mid N$ and elements $a_0, a_1,\dots,a_{p-1}\in A$ such that
\begin{equation}\label{cub-e66}
a_\nu\in \Lambda(y_\nu,pD(N))\hbox{ for }\nu=0,1,\dots,p-1,
\end{equation}
where $y_\nu:= a_0+\nu N/p$.
\end{corollary}

\begin{proof}
We induct on $K$. For $K=1$ and $p_1=p$, if $\Phi_{p^\alpha} \mid A$ for some $1\leq \alpha\leq n_1$, then $A$ is $p$-fibered on that scale, which clearly implies the conclusion.

Assume now that $K\geq 2$ and that the corollary is true with $K$ replaced by $K-1$. Let $N$ and $A$ be as in the statement of Corollary \ref{cor-split}. For $a\in A$, consider the flat cuboids
$$
\Delta_{a;d_1,\dots,d_{K-1}} (X) : = X^{a} \prod_{1\leq j\leq K-1} \big(1-X^{d_jN/p_j}\big) \mod X^{N} - 1,
$$
with $(d_j,p_j)=1$ and $y_\nu:= a+\nu N/p$. We consider two cases.

\begin{itemize}

\item Suppose first that $\mathbb{A}^N[\Delta_{a;d_1,\dots,d_{K-1}}]   \neq 0 $ for some $a\in A$ and $d_1,\dots,d_{K-1}$ as above. Let $y_\nu:= a+\nu N/p$ for $\nu\in\{0,\dots,p_K-1\}$. By Lemma \ref{lma:cuboidstackflat}, we have
$\mathbb{A}^N[\Delta_{y_\nu;d_1,\dots,d_{K-1}}]   \neq 0   $ for all $\nu$, and the conclusion holds with $a_0=a$ and $p=p_K$.

\item Assume now that $\mathbb{A}^N[\Delta_{a;d_1,\dots,d_{K-1}}] = 0   $ for all $a\in A$ and for all $d_1,\dots,d_{K-1}$ as above. Recall that $M_K=M/p_K^{n_K}$, so that $N':=(N,M_K)$ is relatively prime to $p_K$.
Let $A':=A\cap \Lambda(a,D(N'))$ for some $a\in A$. By Lemma \ref{lma:chaindivflat}, we have
$\Phi_{N'} \mid A'$. Since $N'$ has only $K-1$ distinct prime divisors, we may apply the inductive assumption to $A'$ and conclude that (\ref{cub-e66}) holds with $p=p_i$ for some $j\in\{1,\dots,K-1\}$.

\end{itemize}

\end{proof}


\section{Long fibers}
\label{sec-long-fibers}

In this section, we prove a multiscale generalization of the de Bruijn-R\'edei-Schoenberg structure theorem for vanishing sums of roots of unity (Proposition \ref{cuboid}).
Instead of assuming that $A$ has just one cyclotomic divisor, we will assume that $\Phi_L\mid A$ for all $L$ such that $N\mid L \mid M$ for some fixed $N \mid M$. Under that assumption, we prove that we can express $A$ as a linear combination of ``long fibers", which we now define.

\begin{definition}\label{defn:longfibers} {\bf (Long fibers)}
Let $M = \prod_{i=1}^{K} p_i^{n_i}$,  and let $1 \leq \alpha \leq n_i $. We say that a set $F \subset \mathbb{Z}_M$ is a $p_i^\alpha$-\emph{fiber on scale }$M$ if $F \equiv x * F_{i,\alpha}$ mod $M$ for some $x \in \mathbb{Z}_M$, where
$$
F_{i, \alpha} (X) : = \prod_{\nu = 1}^{\alpha} \Phi_{p_i} \big(X^{M/p_{i}^{\nu}}\big) \equiv \frac{X^M - 1}{X^{M/p_i^{\alpha }} - 1}.
$$
We will often refer to $p_i^\alpha$ fibers with $\alpha>1$ as \emph{long fibers} in the $p_i$ direction.
\end{definition}

Explicitly, we have
\begin{equation}\label{long-fiber-long}
F_{i,\alpha } (X)  = 1 + X^{M/p_i^\alpha} +X^{2M/p_i^\alpha} + \cdots + X^{(p_i^\alpha  -1)M/p_i^\alpha}.
\end{equation}
In particular, when $\alpha=1$, the sets $x * F_{i,1}$ are the usual fibers in the $p_i$ direction on scale $M$. The following simple result concerning the cyclotomic divisors of long fibers follows immediately from the definition.

\begin{lemma}\label{lma:longfibersdivisors}
Let $M$ and $\alpha$ be as in Definition \ref{defn:longfibers}.  Then $\Phi_L (X) \mid F_{i,\alpha} (X)$ if and only if $p_i^{n_i - \alpha+1} \mid L \mid M$.
\end{lemma}

\begin{proposition}\label{lma:longfiberslincom}
{\bf(Long fiber decomposition)}
Let $M = \prod_{i=1}^{K} p_i^{n_i}$, and let $N \mid M$ satisfy $N = \prod_{i = 1}^{K} p_i^{n_i - \alpha_i+1}$ with $1 \leq  \alpha_i \leq n_i$. Let $A\in \calm(\ZZ_M)$, and assume that $\Phi_{L} (X) \mid A(X)$ for each $N \mid L \mid M$.  Then, there exist polynomials $P_i (X) \in \mathbb{Z} [X]$ such that
\begin{equation}\label{eq:longfiberslincom}
A(X) = P_1 (X) F_{1, \alpha_1} (X) + \cdots + P_K (X) F_{K, \alpha_K} (X) \mod X^M - 1.
\end{equation}
Moreover, if $A \in \calm^+ (\mathbb{Z}_M)$ and $K = 2$, then we may assume that the polynomials $P_1  (X)$ and $P_2 (X)$ each have non-negative coefficients.
\end{proposition}

\begin{proof}
Let $G(X) : = \prod_{L : N \mid L \mid M} \Phi_L (X)$. Because each of the polynomials $\Phi_L (X)$ is irreducible in $\mathbb{Q}[X]$, we know that $G(X) \mid A(X)$. Hence, it suffices to show that $G$ satisfies \eqref{eq:longfiberslincom} for some polynomials $P_1,\ldots,P_K$ with integer coefficients.

\noindent
We first use Lemma \ref{lma:longfibersdivisors} to observe that
$
G(X) = \gcd (F_{1,\alpha_1},\ldots,F_{K,\alpha_K})
$
in the polynomial ring $\mathbb{Q}[X]$. This follows since
\begin{eqnarray*}
\Phi_L (X) \mid G(X) & \Leftrightarrow & N \mid L \mid M  \\[1ex]
\quad &\Leftrightarrow & L \mid M \textrm{ while } L \nmid \frac{M}{p_i^{\alpha_i}} \textrm{ for all } i = 1,\ldots,K.
\end{eqnarray*}
So, there necessarily exist polynomials $P_1,\ldots,P_K \in \mathbb{Q}[X]$ such that
\begin{equation}\label{eq:pregauss}
G(X) = P_1 (X) F_{1,\alpha_1} (X) + \cdots + P_K (X) F_{K,\alpha_K} (X) \mod X^M - 1.
\end{equation}
Moreover, the polynomials $F_{1,\alpha_1},\ldots,F_{K, \alpha_K}$ all have integer coefficients and are primitive (since they are the product of cyclotomic polynomials, which are monic polynomials). Hence, an application of Gauss's Lemma (see \cite[Chapter 3]{winkler}) implies that each of the polynomials $P_1,\ldots,P_K$ in \eqref{eq:pregauss} can be taken to have integer coefficients. Again, using that $G(X)$ is a divisor of $A(X)$, we obtain that $A$ is, itself, expressible as a linear combination of long fibers with integer coefficients.

It remains to show when $K = 2$ (so that $M = p_1^{n_1}p_2^{n_2}$ and $N = p_1^{n_1 - \alpha_1 + 1} p_2^{n_2 - \alpha_2 +1}$) that we can take $P_1$ and $P_2$ in \eqref{eq:longfiberslincom} to have non-negative integer coefficients. For this, we adapt the proof of Proposition 3.8(b) of \cite{KMSV} to the setting of long fibers.

We can assume that $A = A \cap \Lambda$ where $\Lambda = \Lambda (x,D(N))$ for some $x \in \mathbb{Z}_M$. This follows because, for each $N \mid L \mid M$, Lemma \ref{grid-split} guarantees that $\Phi_L (X) \mid (A \cap \Lambda)(X)$. To further simplify, let us assume that $x = 0$. We now make the observation that (as multisets) $\Lambda = F_{1,\alpha_1} * F_{2,\alpha_2}$. This follows because
$$
(F_{1,\alpha_1} *F_{2,\alpha_2})(X) = F_{1,\alpha_1} (X) F_{2,\alpha_2} (X) = \bigg(\prod_{\nu = 1}^{\alpha_1} \Phi_{p_1} \big(X^{M/p_{1}^{\nu}}\big) \bigg) \bigg(\prod_{\nu = 1}^{\alpha_2} \Phi_{p_2} \big(X^{M/p_{2}^{\nu}}\big)\bigg)
$$
so that $\gcd (M/p_1^{\alpha_1}, M/p_2^{\alpha_2}) \mid y$ for every $y \in F_{1,\alpha_1} * F_{2,\alpha_2}$. But
$$
\gcd\left(\frac{M}{p_1^{\alpha_1}}, \frac{M}{p_2^{\alpha_2}} \right) = p_1^{n_1-\alpha_1} p_2^{n_2-\alpha_2} = \frac{N}{p_1 p_2} = D(N),
$$
and so $F_{1,\alpha_1} *F_{2,\alpha_2} \subset \Lambda$.
The reverse containment follows similar reasoning, and is a consequence of the Chinese Remainder Theorem. We leave the details to the reader.

Having now shown that $\Lambda (0,D(N)) = F_{1,\alpha_1} * F_{2,\alpha_2}$ while also showing that we can assume that $A \cap \Lambda(0,D(N)) = A$. Together, this implies that there exist integers $w_s, v_t$ such that
\begin{align}\label{eq:convolutiongrid}
A(X) = \sum\limits_{s \in F_{2,\alpha_2}} w_sX^s F_{1, \alpha_1} (X) + \sum\limits_{t \in F_{1,\alpha_1}} v_t X^t F_{2,\alpha_2} (X)  &\mod X^M - 1,
\end{align}
and that $w_s + v_t \geq 0$ whenever
\begin{equation}\label{eq:longfiberscross}
(\{s\} * F_{1,\alpha_1}) \cap (\{t\} * F_{2,\alpha_2}) \neq \emptyset.
\end{equation}
We claim that there exists a modification $w_s',v_t'$ of each of these coefficients $w_s, v_t$ such that $w_s'+ v_t' = w_s + v_t$ for all pairs of $s,t$ as in \eqref{eq:longfiberscross} and so that also $w_s', v_t' \geq 0$ for every $s,t$.

To this end, let
$
e : = \min_{x \in \Lambda} w_A^M (x)
$
be the minimal weight of the multiset $A$ in $\mathbb{Z}_M$ and choose $s_0 \in F_{2, \alpha_2}$ and $t_0 \in F_{1,\alpha_1}$ so as to satisfy
$
e = w_{s_0} + v_{t_0}.
$
We then let $w_s' : = w_s + v_{t_0}$ and $v_t' = (w_{s_0} + v_t) - e$ for each $s \in F_{2,\alpha_2}$ and $t \in F_{1,\alpha_1}$. Clearly, then, $v_t' \geq 0$ for every $t \in F_{1,\alpha_1}$, since $e$ was chosen to be the minimal value of the weights of $A$. Moreover, $w_s' \geq 0$, since
$$
w_s' = w_s + v_{t_0} \geq w_{s_0} + v_{t_0} = e
$$
and $e \geq 0$, since we are assuming that $A$ is a non-negative multiset. Finally, we notice that
$$
w_s' + v_t' = \big((w_{s_0} + v_t) - e\big) + w_s + v_{t_0} = w_s + v_t.
$$
Hence, letting
$$
P_1 (X) : = \sum\limits_{s \in F_{2,\alpha_2}} w_s' X^s, \, P_2 (X) : = \sum\limits_{t \in F_{1,\alpha_1}} v_t' X^t,
$$
we see from \eqref{eq:convolutiongrid} that
$$
A(X) = P_1 (X) F_{1,\alpha_1} (X) + P_2 (X) F_{2,\alpha_2} (X) \mod X^M - 1,
$$
and that $P_1, P_2 \in \mathbb{Z}[X]$ both have non-negative coefficients.
\end{proof}


\section{Fibered lower bound: positive results}\label{sec-fib-bound-works}

\subsection{Two prime divisors}\label{sec-2primes}

In this section we work in $\ZZ_M$, where $M=p_1^{n_1}p_2^{n_2}$ has two distinct prime divisors. To simplify the notation, we will abbreviate $p:=p_1$ and $q:=p_2$. We will continue to use the numerical indices where appropriate, so that for example $F^N_1$ will still denote a fiber in the $p$ direction on scale $N$, $F_2^N$ will dennote a fiber in the $q$ direction, and $\capexp_1(S)$ will continue to denote the set of exponents of $p=p_1$ in $S$. We recall here that the exponent sets $\capexp_i(S)$ were defined in (\ref{eq-defexp}).

We first mention a special case resolved in \cite{LM}.

\begin{theorem}\cite[Theorem 1]{LM}
Let $M = p^{n_1} q^{n_2}$ and $A\in\calm^+(\ZZ_M)$. Let $S\subset\cald(M)$ be nonempty, and assume that $\Phi_s \mid A$ for all $s\in S$. Assume further that $q \nmid \vert A \vert$. Then
$$
|A|\geq p^{E_1}  = \capfib (S, 1)  \geq  \capfib (S).
$$
\end{theorem}

In general, the assumption that $q \nmid \vert A \vert$ cannot be dropped. However, we are able to do that in the following special case.

\begin{proposition}\label{prop-expdown}
Let $M = p^{n_1} q^{n_2}$ and $A\in\calm^+(\ZZ_M)$. Let $S\subset\cald(M)$ be nonempty, and assume that $\Phi_s \mid A$ for all $s\in S$.
Assume further that the following holds for all
$\alpha,\alpha'\in \capexp_1(S)$ and $\beta,\beta'\in \capexp_2(S)$ with $\alpha'<\alpha$ and $\beta'<\beta$:
\begin{equation}\label{e-expdown}
\hbox{ if }p^{\alpha}q^{\beta}\in S\hbox{ and }
p^{\alpha'}q^{\beta'}\not\in S,\hbox{ then }
\{p^{\alpha'}q^{\beta},
p^{\alpha}q^{\beta'}\}\cap S\neq\emptyset.
\end{equation}
Then $|A|\geq\capfib(S)$.

\end{proposition}

\begin{proof}
We proceed by induction on $|S|$. The case $|S|=1$ follows from Lemma \ref{structure-thm}.
Assume now that $|S|\geq 2$, and that the lemma is true with $S$ replaced by any $S'$ such that $|S'|<|S|$.
Let $\alpha_0=\min (\capexp_1(S))$ and $\beta_0=\min (\capexp_2(S))$.
Let $A\in\calm^+(\ZZ_M)$ satisfy $\Phi_s \mid A$ for all $s\in S$.

\noindent{\bf Case 1.}
Assume first that there exists some $s\in S$  such that $p^{ \alpha_0}\parallel s$ and $A$ mod $s$ contains a fiber $a*F_1^{s}$. For each $\nu\in\{0,1,\dots,p-1\}$, let
$a_\nu\in A$ satisfy $a_\nu\equiv a+\nu s/p$ mod $s$, and let
$$
A_\nu:= A\cap \Lambda(a_\nu,p^{\alpha_0}).
$$
Let $\pi_{s'}$ denote the natural projection from $\ZZ_M$ to $\ZZ_{s'}$.
Then each $A_\nu(X)$ is divisible by $\Phi_{s'}$ for all $s'\in S':=\{s'\in S:\ p^{\alpha_0+1} \mid s'\}$ since  $\pi_{s'}(A_{\nu}(X))$ is the union of $D(s')$-grids for $s' \in S'$ so the claim follows from Lemma \ref{grid-split}. By the inductive assumption, there are assignment functions $\sigma_\nu:S'\to\{1,2\}$ such that $|A_\nu| \geq \capfib (S',\sigma_\nu)$. We now define $\sigma$ as follows:
choose $\mu\in\{1,\dots,p_i\}$ such that $ \capfib (S',\sigma_\mu)$ is smallest.
For $s'\in S'$, we let $\sigma(s'):=\sigma_\mu(s')$ with that $\mu$. We complete the choice of $\sigma$ by letting $\sigma(s'):=1$ for all $s'$ with $p^{\alpha_0}\parallel s'$.  Then
$$
|A|\geq \sum_\nu |A_\nu| \geq p\cdot\capfib (S',\sigma_\mu) = \capfib (S,\sigma).
$$

Clearly, the same argument works if the assumptions of Case 1 are satisfied with $p$ and $q$ interchanged (so that for some $s\in S$ we have
$q^{ \beta_0}\parallel s$ and $A$ mod $s$ contains a fiber $a*F_2^{s}$). Furthermore, if
$s_0:=p^{\alpha_0}q^{\beta_0}\in S$, then the assumptions of Case 1 hold with $s=s_0$ for some permutation of $p$ and $q$, since $A$ mod $s_0$ has to contain a fiber in at least one direction.

\noindent{\bf Case 2.}
We now consider the complementary case when $s_0\not\in S$ and:
\begin{itemize}
\item[(i)] if $s\in S$ and $p^{\alpha_0}\parallel s$, then $A$ mod $s$ is fibered in the $q$ direction,

\item[(ii)] if $s\in S$ and $q^{\beta_0}\parallel s$, then $A$ mod $s$ is fibered in the $p$ direction,
\end{itemize}

It follows from (i) and (ii) that $|A|$ is divisible by $p^mq^n$, where $m=\#\{s\in S:\ q^{\beta_0}\parallel s\}$ and $n=\#\{s\in S:\ p^{\alpha_0}\parallel s\}$

We define an assignment function as follows. Let $s=p^\alpha q^\beta \in S$. By (\ref{e-expdown}) with $\alpha'=\alpha_0$ and $\beta'=\beta_0$, at least one of
$p^{\alpha_0}q^{\beta},
p^{\alpha}q^{\beta_0}$ is in $S$. We let $\sigma(s)=1$ if $p^{\alpha}q^{\beta_0}\in S$, and $\sigma(s)=2$ if $p^{\alpha}q^{\beta_0}\not\in S$ but $p^{\alpha_0}q^{\beta}\in S$. Then
$$
\capfib(S;\sigma) = p^mq^n\mid  |A|,
$$
and the proposition is proved.

\end{proof}

\begin{corollary}\label{cor-2exponents}
Let $M = p^{n_1} q^{n_2}$ and $A\in\calm^+(\ZZ_M)$. Let $S\subset\cald(M)$ be nonempty, and assume that $\Phi_s \mid A$ for all $s\in S$. Assume further that $|\capexp_i(S) |\leq 2$ for some $i\in\{1,2\}$. Then
$|A|\geq\capfib(S)$.

\end{corollary}

\begin{proof}
Without loss of generality, we may assume that $i=2$. If $|\capexp_2(S) |=1$, then
 (\ref{e-expdown}) holds vacuously. (In the language of the proof of Proposition \ref{prop-expdown}, we are in Case 1 throughout the inductive argument.)

 Assume now that $|\capexp_2(S) |=2$. We need to prove that (\ref{e-expdown}) holds in this case. Let $\alpha,\alpha'\in \capexp_1(S)$ and $\beta,\beta'\in \capexp_2(S)$ with $\alpha'<\alpha$ and $\beta'<\beta$. Note that we must have $\capexp_2(S)=\{\beta',\beta\}$. Therefore, if $p^{\alpha'}q^{\beta'}\not\in S$, we must have $p^{\alpha'}q^{\beta}\in S$, and (\ref{e-expdown}) holds again.
 \end{proof}

\begin{proposition}\label{3divisors}
Let $A\in\calm^+(\ZZ_{M})$ with $M=p^{n_1}q^{n_2}$. Assume that $\Phi_{p^{m_1}}\cdots \Phi_{p^{m_r}}\Phi_{p^\alpha q^\beta}\Phi_{q^\gamma} \mid A$ for some $1\leq \alpha<m_1<\dots<m_r\leq n_1$ and $1\leq\beta,\gamma \leq n_2$. Assume further that $\beta\neq\gamma$. Then $|A|\geq p^rq\min(p,q)$. In particular, $|A|\geq\capfib(S)$.
\end{proposition}

We clearly have $|A|\geq p^rq$ based on $\Phi_{p^{m_1}}\cdots \Phi_{p^{m_r}}\Phi_{q^\gamma}|A$, and independently, $|A|\geq\min(p,q)$ based on $\Phi_{p^\alpha q^\beta}|A$ by either (\ref{e-lamleung}) or Lemma \ref{structure-thm}. The point of the lemma is that these bounds boost each other.

\begin{proof}[Proof of Proposition \ref{3divisors}]
By Proposition \ref{prop:truncation}, there exists a multiset $\tilde{A}\in\calm^+(\ZZ_{p^{r+1}q^2})$ such that $|\tilde{A}|=|A|$ and
$$\Phi_{p^2} \cdots \Phi_{p^{r+1}} \Phi_{pq^{\beta'}}\Phi_{q^{\gamma'}}\mid \tilde{A}
$$
for some $\beta',\gamma'$ such that $\{\beta',\gamma'\}=\{1,2\}$.
It suffices to prove that $|\tilde{A}|\geq p^r q\min(p,q)$. To simplify the notation, we may assume in the rest of the proof that $\tilde{A}=A$, $\alpha=1$, $m_j=j+1$ for $j=1,\dots,r$, and $\{\beta,\gamma\}=\{1,2\}$.

We first note that
\begin{equation}\label{tilde-no-tilde}
p^r q = \Phi_{p^2}(1) \cdots \Phi_{p^{r+1}}(1)\Phi_{q^\gamma}(1) \mid A(1)=|A|.
\end{equation}

Let $s=pq^\beta$. Since $\Phi_{s} \mid A$, by Lemma \ref{structure-thm} $A$ mod $s$ is a linear combination, with nonnegative coefficients, of $p$-fibers and $q$-fibers on scale $s$. If $A$ is $q$-fibered on scale $s$, then we actually have $\Phi_q\Phi_{q^2}\mid A$ and $q^2\mid |A|$, so that we are done in this case.

It remains to consider the case when $A$ mod $s$ has at least one $p$-fiber on that scale.
Thus  $A(X)=A'(X)+A''(X)$, where $A',A''\in \calm^+(\ZZ_M)$
are multisets such that $A'$ mod $s$ is fibered in the $p$ direction, $A''$ mod $s$ is fibered in the $q$ direction, and $A'\neq\emptyset$.

By definition, we have $p\mid |A'|$ and $q\mid |A''|$. It follows from (\ref{tilde-no-tilde}) that
\begin{equation}\label{toy-e1}
q\mid |A'| \hbox{ and }p\mid |A''|.
\end{equation}
so that $A'$ mod $s$ is a union of $kq$ many fibers in the $p$ direction for some $k\neq 0$.

For $j=0,1,\dots,p-1$, let
$$
R_j:=\{x\in\ZZ_M: x\equiv j\bmod p\},\ A_j:=  A\cap R_j.
$$
Then
$$
A_j=(A'\cap R_j)\cup (A''\cap R_j)
$$
and $A''\cap R_j$ is $q$-fibered on scale $s$. Hence $|A_j|=kq+\ell_j q$, where $\ell_j$ is the number of $q$-fibers in $A_j$ that are also $q$-fibers of $A''$ on scale $s$.

Since $\Phi_{p^2}\dots \Phi_{p^{r+1}} \mid  A$, by Lemma \ref{grid-split} we must have $\Phi_{p^2}\dots \Phi_{p^{r+1}} \mid A_j$ for each $j$. In particular, $p^r\mid |A_j|$. Together with the above, this yields $p^r q\mid |A_j|$ for each $j$. Since $k>0$, we also have $|A_j|>0$ for each $j$. Therefore
$$
|A|=\sum_{j=0}^{p-1} |A_j|\geq p^{r+1}q
$$
as claimed.

To complete the proof of the proposition, we need to prove that $p^rq\min(p,q)\geq\capfib(S)$.
Indeed, let an assignment function $\sigma$ satisfy  $\sigma(p^{m_j})=1$ for $j=1,\dots,r$ and $\sigma(q^\gamma)=2$, and define $\sigma(p^\alpha q^\beta)$ so that $p_{\sigma(s_3)}=\min(p,q)$.
Then $p^rq\min(p,q)=\capfib(S,\sigma)$.
\end{proof}

An analogous result can be proved if the exponents of the $p$-power cyclotomic divisors are smaller than $\alpha$.

\begin{proposition}\label{3.2divisors}
Let $A\in\calm^+(\ZZ_{M})$ with $M=p^{n_1}q^{n_2}$. Assume that $\Phi_{p^{m_1}}\cdots \Phi_{p^{m_r}}\Phi_{p^\alpha q^\beta}\Phi_{q^\gamma} \mid A$ for some $1\leq m_1<\dots<m_r<\alpha\leq n_1$ and $1\leq\gamma<\beta \leq n_2$.  Then $|A|\geq p^rq\min(p,q)$.
In particular, $|A|\geq\capfib(S)$.

\end{proposition}

\begin{proof}
  By Proposition \ref{prop:truncation}, there exists a multiset $\tilde{A}\in\calm^+(\ZZ_{p^{r+1}q^2})$ such that $|\tilde{A}|=|A|$ and
$$\Phi_{p} \cdots \Phi_{p^{r}} \Phi_{p^{r+1}q^2}\Phi_{q}\mid \tilde{A}.$$
It suffices to prove that $|\tilde{A}|\geq p^r q\min(p,q)$. To simplify the notation, we may assume in the rest of the proof that $\tilde{A}=A$, $\alpha=r+1, \beta=2, \gamma=1$, $m_j=j$ for $j=1,\dots,r$, and $n_1=r+1, n_2=2$.

Let $\Lambda_\nu=\Lambda(\nu,p^{r})$ for $\nu=0, 1,\dots,p^{r}-1$.
Since $\Phi_p\cdots \Phi_{p^r}\mid A$, we have
$$\left|A\cap \Lambda_\nu \right|=\frac{|A|}{p^r}\hbox{ for all }\nu.$$
Since $\Phi_q\mid A$, we have $q\mid |A|$ and so $q\mid \frac{|A|}{p^r}=|A\cap \Lambda_\nu|$.
Finally, we have $\Phi_{p^{r+1}q^2}\mid A$. By Lemma \ref{structure-thm}, this implies that $A$ (hence also $A\cap \Lambda_\nu$ for each $\nu$) is the sum of $p$-fibers and $q$-fibers on that scale. Hence for each $\nu$ we have $|A\cap \Lambda_\nu|=kp+lq$ for some $k,l\ge 0$ (possibly depending on $\nu$).

If $k=0$ for all $\nu$, then $A$ is the union of $q$-fibers only, meaning that $\Phi_{q^2}\mid A$ and hence $\Phi_{p} \cdots \Phi_{p^{r}} \Phi_{q}\Phi_{q^2}\mid A$. This implies that $p^rq^2\mid |A|$, and we are done.

Suppose now that $k>1$ for some $\nu$. Then $q\mid k$, since $q\mid |A\cap \Lambda_\nu|=kp+lq$. Hence $|A \cap \Lambda_\nu| \ge pq$, and thus $$|A|=p^r\cdot \left|A \cap \Lambda_\nu \right|\ge p^{r+1}q,$$ and we are done again.

The last statement in the proposition
follows as in the proof of Proposition \ref{3divisors}.
\end{proof}

\begin{theorem}\label{thm-3divisors-ZM}
Let $M=p^{n_1}q^{n_2}$, and let $S\in\cald(M)$ satisfy $|S|=3$. Then $\capmin(S)= \capfib(S)$.
\end{theorem}

\begin{proof}
If $|\capexp_i(S)|\leq 2$ for some $i\in\{1,2\}$, then the result follows from Corollary \ref{cor-2exponents}. We may therefore assume for the rest of the proof that
$$
S=\{s_1,s_2,s_3\}, \hbox{ where }s_i=p^{\alpha_i}q^{\beta_i},\ i=1,2,3,
$$
where $0\leq \alpha_1<\alpha_2<\alpha_3\leq n_1$, and where the exponents $\beta_1,\beta_2,\beta_3\in\{0,1,\dots,n_2\}$ are all distinct.

Let $A\in\calm^+(\ZZ_M)$ satisfy $\Phi_s\mid A$ for all $s\in S$.
Suppose that $\alpha_1\geq 1$, and that $A$ mod $s_1$ contains a fiber in the $p$ direction.
By the same argument as in the proof of Proposition \ref{prop-expdown}, Case 1, we reduce the proof of the theorem in this case to proving that $\capmin(\{s_2,s_3\})=\capfib(\{s_2,s_3\})$; however, for a 2-element set $\{s_2,s_3\}$, this equality again follows from Corollary \ref{cor-2exponents}.

If on the other hand $\alpha_1\geq 1$ and $A$ mod $s_1$ is fibered in the $q$ direction, then $A(X)$ is also divisible by $\Phi_{q^{\beta_1}}$. We replace the set $S$ with the set $S':=\{s'_1,s_2,s_3\}$, where $s'_1=q^{\beta_1}$, note that $\Phi_s\mid A$ for all $s\in S'$, and continue with the rest of the proof.

Similarly, let $\mu\in\{1,2,3\}$ be the index such that $\beta_\mu=\min(\beta_1,\beta_2,\beta_3)$. If $A$ mod $s_\mu$ contains a fiber in the $q$ direction, we proceed as in the proof of Proposition \ref{prop-expdown}, Case 1, to reduce $S$ to a 2-element set covered in Corollary \ref{cor-2exponents}. If on the other hand $A$ mod $s_{\mu}$ is fibered in the $p$ direction, then $A(X)$ is also divisible by $\Phi_{p^{\alpha_\mu}}$, so that we may also replace $s_\mu$ by $p^{\alpha_\mu}$.

By the above reductions, we may assume for the rest of the proof that
$$
0= \alpha_1<\alpha_2<\alpha_3\leq n_1, \ \ \min(\beta_1,\beta_2,\beta_3)=\beta_\mu=0.
$$
Since $1\not\in S$, we have $\mu\in\{2,3\}$.

Assume first that $\beta_3=0$. Then
$$
\Phi_{q^{\beta_1}}\Phi_{p^{\alpha_2}q^{\beta_2}} \Phi_{p^{\alpha_3}}\mid A,
$$
with $0< \alpha_2<\alpha_3$. This places us in the situation described in Proposition \ref{3divisors} with $r=1$. Let $\sigma(s_1)=2$,  $\sigma(s_3)=1$, and define $\sigma(s_2)$ so that $p_{\sigma(s_2)}=\min(p,q)$. By
Proposition \ref{3divisors}, we have
$$
|A|\geq pq\min(p,q)=\capfib(S,\sigma)\geq \capfib(S),
$$
and we are done in this case.

It remains to consider the case when $\beta_2=0$. Thus
$$
\Phi_{q^{\beta_1}}\Phi_{p^{\alpha_2}}\
\Phi_{p^{\alpha_3}q^{\beta_3}} \mid A.
$$
If $0<\beta_3<\beta_1$, we are again in the situation described in Proposition \ref{3divisors}, but with $p$ and $q$ interchanged. The theorem follows as above.

We are left with the case when
        $$
0= \alpha_1<\alpha_2<\alpha_3, \ \
0=\beta_2<\beta_1<\beta_3.
$$
This is a special case of Proposition \ref{3.2divisors} with $r=1$,
and we are done again.
\end{proof}

\subsection{The diagonal case}

We return to the setting where $M$ has arbitrarily many prime divisors, and consider the following simple case.

\begin{lemma}\label{lemma-diagonal-straight}
Let $M = \prod_{k=1}^{K} p_k^{n_k}$. Assume that $S=\{s_1,\dots,s_m\}\subset \cald(M)$ satisfies
\begin{equation}\label{e-diagonal-straight}
s_j\mid D(s_{j+1})\hbox{ for }j=1,\dots,m-1.
\end{equation}
Then $\capmin(S)\geq \prod_{j=1}^m \min_{i:p_i \mid s_j} p_i = \capfib(S)$.

\end{lemma}

\begin{proof}
We proceed by induction in $m$. If $m=1$ and $S=\{s\}$, then the conclusion follows from (\ref{e-lamleung}).
Assume now that $m>1$, and that the conclusion is true when $|S|=m-1$.
Let $A\in\calm^+(\ZZ_M)$ satisfy $\Phi_{s_j} \mid A$ for $j=1,\dots,m$. Since $\Phi_{s_1} \mid A$, we have that $A$ mod $s_1$ satisfies the conclusion (\ref{cub-e66}) of Corollary \ref{cor-split} for some $a\in A$ and $p\mid s_1$.
Fix that $p$.
For $\nu=0,1,\dots,p-1$, let $y_\nu:= a+\nu s_1/p$ and
$$
A_\nu:=A\cap\Lambda(y_\nu,pD(s_1)).
$$
Then the sets $A_\nu$ are nonempty and pairwise disjoint. By (\ref{e-diagonal-straight}), we have $pD(s_1)\mid D(s_j)$ for $j=2,\dots,m$. It follows from Lemma \ref{grid-split} that $\Phi_{s_j} \mid A_\nu$ for all $j=2,\dots,m$ and $\nu=1,\dots,p$. Applying the inductive assumption to $A_\nu$, we get
$$
|A|\geq \sum_{j=1}^p |A_\nu|\geq p  \prod_{j=2}^m \min_{i:p_i \mid s_j} p_i
\geq  \prod_{j=1}^m \min_{i:p_i \mid s_j} p_i.
$$

By Proposition \ref{prop-multifibered} we have that \begin{equation}\label{best-we-can-hope-for-2}
\capfib(S)= \min_{\sigma} \capfib (S,\sigma),
\end{equation}
with the minimum taken over all assignment functions $\sigma$. This minimum is clearly taken for the assignment function $\sigma$ defined via  $\sigma(s)=p_s$ for every $s\in S$, where  $p_s$ is the smallest prime divisor of $s$.  This implies that $$\prod_{j=1}^m \min_{i:p_i \mid s_j} p_i=\capfib(S).$$
\end{proof}


\subsection{Many Primes with a Growth Condition}
\begin{theorem}\label{thm:separated3prime}
Let $A \in \calm^+ (\ZZ_M)$ with $M = \prod_{k=1}^{K} p_k^{n_k}$ satisfying
\begin{equation}\label{e-growth}
p_K > \cdots > p_2 > p_1^{n_1}.
\end{equation}
Let $S\subset\cald(M)$ be nonempty, and assume that $\Phi_s \mid A$ for all $s\in S$.
Then $|A| \geq p_1^{E_1}$, where
$$
E_1 : = \# \textsf{EXP}_1(S) : = \# \{\alpha \geq 1 : \exists \, s \in S \textrm{ with } p_1^{\alpha} \mid \mid s \}.
$$
\end{theorem}

In particular, we have $|A|\geq\capfib(S)$ under the assumptions of the theorem, so that $\capmin(S)\geq \capfib(S)$ if (\ref{e-growth}) holds. Indeed, let $\sigma:S\to\{1,\dots,K\}$ be an assignment function. If $\sigma(s)\geq 2$ for any $s\in S$, we have $\capfib(S,\sigma)\geq p_2>p_1^{n_1}\geq p_1^{E_1}$. If on the other hand $\sigma(s)=1$ for all $s\in S$ (note that this can only happen when $p_1 \mid s$ for all $s\in S$), then $\capfib(S,\sigma)=p_1^{E_1}$. The claim follows in both cases.

\begin{proof}[Proof of Theorem \ref{thm:separated3prime}]
We induct on $K$. When $K = 1$,  we have $S= \{p_1^{\alpha_1},\ldots,p_1^{\alpha_{E_1}} \}$ for some $1\leq \alpha_1<\dots,<\alpha_{E_1}\leq n_1$, so that
$$
\Phi_{p_1^{\alpha_1}} \cdots \Phi_{p_1^{\alpha_{E_1}}} (X) \mid A(X) \Rightarrow p_1^{E_1} \mid \vert A \vert,
$$
and so we clearly have $\vert A \vert \geq p_1^{E_1}$.   Assume now that $K \geq 2$ and that the result holds for any $1 \leq K_0 \leq K - 1$.  Let
$$
\mathcal{C} : = \{s \in S_A : p_K \mid s \},\ \ M' : = M/p_{K}^{n_K} = p_1^{n_1} \cdots p_{K-1}^{n_{K-1}},
$$
and note that the size assumption $p_{K-1} > \cdots > p_1^{E_1}$ is still satisfied for $M'$ (vacuously if  $K=2$).
There are two cases to consider, corresponding to the conclusions (1) and (2) of Proposition \ref{lma:reductionindimension}.

\medskip\noindent
{\bf Case 1}: We first assume that the conclusion of Proposition \ref{lma:reductionindimension} (1) fails for some $s \in \mathcal{C}$. Assume that $p^{\alpha}_K \parallel s$. Then the failure of (1) means that
there exists $1 \leq \beta \leq \alpha$ such that $\Phi_{s/p_K^{\beta}} \nmid A$.
Applying Proposition \ref{lma:reductionindimension} (2) with $N=s$ and $p=p_K$, we find $a \in A$ such that
$$
\vert A \vert = A(1) \geq \sum\limits_{\nu = 0}^{p_K-1} A_{\nu,a} (1) \geq p_K
$$
where each $A_{\nu,a} \in \calm (\mathbb{Z}_N)$ is as in \eqref{eq:Aplanesandgrids}.
Since we assume that $p_K > p_1^{n_1} \geq p_1^{E_1}$,  this establishes Theorem \ref{thm:separated3prime} in this case.
\medskip

\noindent
{\bf Case 2}: We now suppose that for all $s \in \mathcal{C}$, we necessarily have that
\begin{equation}\label{eq:chaindivisors}
\Phi_s \cdots \Phi_{s/p_K^{\alpha}} \mid A,
\end{equation}
where $\alpha = \alpha(s) \geq 1$ is the unique exponent such that $p_K^{\alpha} \mid \mid s$.
Let
$$
S':= \{(s,M'):\ s\in S\},\ \ A':= (A\bmod M')\in \calm^+ (\mathbb{Z}_{M'}).
$$
By (\ref{eq:chaindivisors}), we have $\Phi_{s'}\mid A$ for all $s\in S'$. Since $S'\subset\cald(M')$, this also implies that
$\Phi_{s'}\mid A$ for all $s\in S'$.
The key observation is that
$$
    \# \{\alpha \geq 1 : \exists \, s \in S'  \textrm{ with } p_1^{\alpha} \mid \mid s \}
 =\# \{ \alpha \geq 1 : \exists \, s \in S \textrm{ with } p_1^{\alpha} \mid \mid s \}= E_1 .
$$
Indeed, if $s\in S\setminus \mathcal{C}$, then $s\in S'$, and if $s\in \mathcal{C}$, then $(s,M')\in S'$ and
$\textsf{EXP}_1 (s) = \textsf{EXP}_1 ((s,M'))$.
Applying the inductive hypothesis to $A'$ and $M'$, we see that
$$
\vert A \vert = \vert A' \vert \geq p_1^{E_1'} = p_1^{E_1}.
$$
\end{proof}

\section{Examples where the fibered lower bound fails}
\label{sec-fib-bound-fails}

\subsection{Recombination effects for $3$ or more prime factors}
In general, if we increase the complexity of $S$, we may have $\capmin(S)<\capfib(S)$. It is easiest to give examples of this when $M:=\lcm(S)$ has 3 or more prime factors. The idea is to use a certain \textit{recombination effect}, as follows. Write $M=PQ$, where $(P,Q)=1$, so that $\ZZ_M=\ZZ_P\oplus\ZZ_Q$. Let $A'\in\calm^+(\ZZ_P)$ and $A''\in\calm^+(\ZZ_Q)$ be two multisets with $|A'|=|A''|$. Then we may construct a multiset $A\in\calm^+(\ZZ_M)$ so that its Chinese Remainder Theorem projections onto $\ZZ_P$ and $\ZZ_Q$ are, respectively, $A'$ and $A''$. While each of $A'$ and $A''$, independently, must have large enough cardinality to accommodate its own cyclotomic divisors, there need not be any additional increases in the size of $A$ due to sharing the cyclotomic divisors of both $A'$ and $A''$.

One example of this, with $M$ equal to a product of 4 primes, is given in \cite[Section 6.3]{LM}. An additional constraint imposed in \cite{LM} (coming from the intended application to the Favard length problem) was that $|A|$ should be relatively prime to $M$. If we drop that constraint, then a simpler example is as follows.

\begin{example}\label{ex-3primes1scale}
\rm{Let $2 = p_1 < p_2 < p_3$ be distinct primes such that
$
p_1+ p_2 = p_3.
$
Let $M=p_1p_2p_3$ and $S = \{p_1p_2, \,  p_3\}$.
Consider any set $A \in\calm^+(\ZZ_M)$ simultaneously satisfying the equations
\begin{equation}\label{ex-basic-1}
\begin{split}
A(X) &\equiv X^{a_1} F_1^{p_1p_2} (X) + X^{a_2} F_2^{p_1p_2}(X) \mod X^{p_1p_2} - 1,\\
A(X) & \equiv X^{a_3} F_3^{p_3} (X) \mod X^{p_3} - 1
\end{split}
\end{equation}
for some $a_1,a_2,a_3\in\ZZ_M$. Such a set can be easily constructed via the Chinese Remainder Theorem. Since the same idea is also used in the more difficult example in Proposition \ref{prop:countex1}, we provide the details as a warm-up.
We recall the array coordinate expansion of elements $x\in\ZZ_M$:
$$
x=x_1M_1+x_2M_2+x_3M_3,
$$
where $x_j\in\{0,1,\dots,p_j-1\}$ and (in this case) $M_j=M/p_j$. Then (\ref{ex-basic-1}) is equivalent to saying that $|A|=p_3$ and
\begin{align*}
\{(a_1,a_2):\ a\in A\}&=\{(0,0),(1,0)\}\cup\{(0,0),(0,1),\dots,(0,p_2-1)\},
\\
\{a_3:\ a\in A\}&=\{0,1,\dots,p_3-1\},
\end{align*}
where the first equation should hold in the sense of multisets, with two different triples of the form $(0,0,a_3)$ in $A$. In each equation above, the cardinality of $A$ matches that of the set on the right side; furthermore, $A$ is a set (not just a multiset) since all its elements are distinct mod $p_3$.
The key point is that the two conditions above involve different coordinates of the elements of $A$, hence they can be imposed independently of each other.

By Proposition \ref{cuboid}, the first equation in (\ref{ex-basic-1}) implies that $\Phi_{p_1p_2}(X)\mid A(X)$, and the second one implies that $\Phi_{p_3} (X) \mid A(X)$. Hence $\textsf{MIN} (S) \leq |A|= p_3$. By (\ref{e-lamleung}), we also have $\textsf{MIN} (S) \geq p_3$ (using that $p_3\in S$), so that $\capmin(S)=p_3$.

On the other hand, we must have $\sigma(p_1p_2)\in\{1,2\}$ and $\sigma (p_3) = 3$ for any assignment function $\sigma : S \rightarrow \{1,2,3\}$.  Hence,
$
\capfib(S)= \min_{\sigma} \capfib (S, \sigma) \geq \min(p_1,p_2) \cdot p_3 = 2  p_3,
$
showing that $\textsf{MIN} (S) < \capfib(S)$.
}
\end{example}

\subsection{Recombination for two prime factors}
\label{sec-recombine-2primes}

More surprisingly, we may have $\capmin(S)<\capfib(S)$ even if $M=\lcm(S)$ has only two distinct prime factors. In this case, it follows from Corollary \ref{cor-2exponents} that, unlike in Example \ref{ex-3primes1scale}, we cannot produce such examples using only two cyclotomic divisors and a single scale. However, we can construct them using multiple scales instead.

\begin{proposition}\label{prop:countex1}
 Let $M=p^nq^m$ with $n\geq 9$ and $m\geq 6$, and let $p=2, q=3 $. Then there exists a set
 $A\subset \Z_M$ such that  $$\Phi_{p^n}\Phi_{p^{n-1}}\Phi_{p^{n-2}}\Phi_{q^{m}}\Phi_{q^{m-1}}\Phi_{q^{m-2}}\Phi_{pq}\mid A$$ and $|A|=p^3q^3$.
 \end{proposition}
\begin{proof}
We first define a multiset $B \in \calm^+(\ZZ_{pq})$ with $p=2, q=3$ via the table below.
It is easy to check explicitly that $\bbB^{pq}[\Delta]=0$ for all $pq$-cuboids $\Delta$ (there are 3 such cuboids).
By Proposition \ref{cuboid}, it follows that
$\Phi_{pq} \mid B$.
\begin{center}
\begin{equation}\label{table1}
\begin{tabular}{ |m{3 cm}|| m{1.5cm} | m{1.5cm} | m{1.5cm} || m {2cm}|}
  \hline
   &0 mod 3 & 1 mod 3 & 2 mod 3 & \mbox{row sum}\\
  \hline
  \hline
  0 mod 2 & 74 & 47 & 47 &  21$\cdot$ 8\\
  \hline
1 mod 2 & 34  & 7 & 7 &   6$\cdot$ 8  \\
  \hline
  \hline
  \mbox{column sum} &4$\cdot$ 27 & 2$\cdot$ 27 & 2$\cdot$ 27&  \\
  \hline
\end{tabular}
\end{equation}
\end{center}
Notice that $27\mid \left|B\cap \Lambda(i, q))\right|$ for every $i\in \{0,1,2\}$. Similarly, $8 \mid |B\cap \Lambda(j, p))|$ for every $j\in \{0,1\}$.

We would like to construct a set $A\subset\ZZ_M$ such that $A\equiv B$ mod $pq$ (so that divisibility by $\Phi_{pq}$ is preserved), but $A$ also has the additional cyclotomic divisors listed in the proposition. Let $M=p^nq^m$, with $n$ and $m$ large enough (to be determined later). We will again use the array coordinates mod $M$: for each $x\in\ZZ_M$ we write
$$
x\equiv x_1M_1+x_2M_2\bmod M,\ \ x_1\in\{0,1,\dots,p^n-1\},\ \ x_2\in\{0,1,\dots,q^m-1\},
$$
where $M_1=q^m=M/p^n$ and $M_2=p^n=M/q^m$. We will further need the digit expansions
$$
x_1=x_{1,0}+x_{1,1}p+\dots+x_{1,n-1}p^{n-1},\ \
x_2=x_{2,0}+x_{2,1}q+\dots+x_{2,m-1}q^{m-1},
$$
where $x_{1,j}\in\{0,1\}$ and $x_{2,i}\in\{0,1,2\}$.

Let $A\subset\ZZ_M$ with $|A|=p^3q^3$; we will now impose conditions on the digits of the elements of $A$ so that $A$ has the required cyclotomic divisors. We first ask that $A\equiv B$ mod $pq$; to this end, it suffices to ensure that the digits $a_{1,0}$ and $a_{2,0}$ of the elements of $A$ have the distribution indicated in the table above.

Next, we need to ensure that $A$ is divisible by
$$
\Phi_{p^n}(X)\Phi_{p^{n-1}}(X)\Phi_{p^{n-2}}(X) = \frac{X^{p^n}-1}{X^{p^{n-3}}-1} = F^{p^n}_{p,3}.
$$
In other words, $A$ mod $p^n$ needs to be a union of long $p^3$-fibers; furthermore, in order for $A$ to be a set, we will make sure that these $p^3$-fibers are disjoint. We write $A=A_0\cup A_1$, where $A_j=\{a\in A:\ a\equiv j\bmod 2\}$.
By (\ref{long-fiber-long}), we have $F^{p^n}_{p,3} = p^{n-3}\ZZ_{p^n}$. Hence, it suffices to have
$$
A_j(X)\equiv X^j C_j(X)F^{p^n}_{p,3} \mod (X^{p^n}-1),\ j=0,1,
$$
where $C_j\subset\ZZ_M$ is a set whose all elements are divisible by $p$ but distinct mod    $p^{n-3}$. This is possible when
\begin{equation}\label{lowexample-e1}
\# \{c_{1,1}p+c_{1,2}p^2+\dots+c_{1,n-4}p^{n-4}:\ c_{1,j}\in\{0,1\}\}\geq 21,
\end{equation}
since we have to place 21 long fibers in $A_0$ and 6 long fibers in $A_1$. Since $32=2^5>21$, it suffices to take $n-4\geq 5$, so that $n\geq 9$. Note that the entire operation above involved only the $a_{1,i}$ digits of $a\in A$ with $i\geq 1$.

To ensure that $\Phi_{q^m}\Phi_{q^{m-1}}\Phi_{q^{m-2}} \mid A$, we proceed similarly, but with $p$ and $q$ interchanged so that we are now adjusting the $a_{2,i}$ digits of $a\in A$ with $i\geq 1$. This can clearly be done independently of the choices already made above. The condition (\ref{lowexample-e1}) is replaced by $3^{m-4}\geq 8$. Since $9=3^2>8$, it suffices to take $m-4\geq 2$, so that $m\geq 6$.
\end{proof}

In the proof of Proposition \ref{prop:countex1}, the
fact that $A$ is a set is guaranteed in the simplest possible way by forcing both $A$ mod ${p^n}$ and $A$ mod ${q^m}$ to be sets. However, it would be enough to assume that one of them is a set (with an arbitrary bijection between this set and the other multiset). It could be possible to lower the exponents $n$ and $m$ further so that both $A$ mod ${p^n}$ and $A$ mod ${q^m}$ are multisets, but then the construction requires further analysis and details.
There is no such construction with $|A|$ equal to  $2^2\cdot 3^3$ or any of its divisors.

We expect that there are other choices of primes (not necessarily $p=2$, $q=3$) for which similar examples could be constructed. However, Proposition \ref{prop-large-p} below shows that $p$ and $q$ cannot be chosen completely arbitrarily in this type of examples.

\begin{proposition}\label{prop-large-p}
Let $1\leq a< n$ and $1\leq b< m$. Assume that $p>q^b$ and
$$\Phi_{p^n}\dots \Phi_{p^{n-a+1}}\Phi_{q^{m}}\dots \Phi_{q^{m-b+1}}\Phi_{pq}\mid A$$
for some $A\in\calm^+(\ZZ_{p^nq^m})$. Then
$|A|>p^aq^b$.
\end{proposition}

\begin{proof}
Based on the prime power cyclotomic divisors, we have
$p^aq^b\mid |A|$, hence $|A| \ge p^aq^b$. Furthermore,
$|A_i|=|A \cap \Lambda(i,p)|$ is divisible by $p^a$ for every $i=0,1,\ldots , p-1$.

Assume indirectly that $|A|=p^aq^b$. Since $p>q^b$, we have that $|A_i|$ is zero for some $i$.  On the other hand, since $\Phi_{pq} \mid A$, by Lemma \ref{structure-thm} we have that $A \bmod{pq}$ is a nonnegative linear combination of $p$-fibers and $q$-fibers. With $A_i=\emptyset$ for some $i$, $A$ must in fact be a sum of $q$-fibers only, so that $\Phi_q \mid A$. It follows that $q^{b+1}\mid A$, contradicting the assumption that $|A|=p^aq^b$.
    \end{proof}


In the next example, the exponents of $p$ and $q$ in the extra composite divisor of $A(X)$ are higher than those in the prime power divisors.

\begin{proposition}\label{prop:countex2}
Let $M=p^4q^4$, $p=2,q=3$. There exists a set $A\subset \Z_M$ such that  $$\Phi_{p}\Phi_{p^{2}}\Phi_{p^{3}}\Phi_{q}\Phi_{q^{2}}\Phi_{M}\mid A$$ and $|A|=p^3q^2=72$.
\end{proposition}
\begin{proof} 
The following table represents a multiset $B\in \mathcal{M}^{+}(\mathbb{Z}_{72})$, where the cyclic group $\Z_{72}$ is written as $\Z_8 \oplus \Z_9$, rows represent cosets of $\Z_9$, and columns represent cosets of $\Z_8$. Similarly to (\ref{table1}), the entry in the $i$-th row and $j$-th column is equal to $w^{72}_B(b)$, where $b$ is the element of $\Z_{72}$ such that $b\equiv i$ mod $8$ and $b\equiv j$ mod $9$.

\begin{center}
\begin{tabular}{ | m{1cm} | m{1cm} | m{1cm} | m{1cm}| m{1cm} | m{1cm} | m{1cm} |m{1cm}| m{1cm} |}
  \hline
  5 & 0 & 0 & 0& 0& 2 &0 &0 &2 \\
  \hline
 3 & 4 & 0 & 0& 0& 2 &0 &0 &0 \\
  \hline
  0 & 0 & 5 & 2& 0& 0 &0 &0 &2 \\
  \hline
  0& 0 & 3 & 2& 0& 0 &0 &4 &0 \\
  \hline
  0 & 0 & 0 & 0& 5& 2 &0 &0 &2 \\
  \hline
  0 & 4 & 0 & 0& 3& 2 &0 &0 &0 \\
  \hline
  0 & 0 & 0 & 0& 0& 0 &5 &2 &2 \\
  \hline
  0 & 0 & 0 & 4& 0& 0 &3 &2 &0 \\
  \hline
\end{tabular}
\end{center}

It is easy to verify that the entries in each column add up to $8$, and the entries in each row add up to 9. In other words, we have $|B\cap \Lambda^{72}(x, 9)|=8$ and $|B\cap \Lambda^{72}(x, 8)|=9$ for all $x\in \Z_{72}$.
This guarantees that
\begin{equation}\label{e-table11}
\Phi_{p}\Phi_{p^{2}}\Phi_{p^{3}}\Phi_{q}\Phi_{q^{2}}\mid B.
\end{equation}

Let $M= p^4q^4$.
We want to construct a set $A\subset \Z_M$ such that $B\equiv A \bmod{p^3q^2}$ and, furthermore, $\Phi_{M}\mid A$.
Let  $\pi: \Z_{M}\to  \Z_{p^3q^2}$ be the natural projection defined by $\pi(x)=x\bmod p^3q^2$. Then for each $x\in\Z_{p^3q^2}$, its preimage $\pi^{-1}(x)=\Lambda(x,p^3q^2)$ contains at least two grids $\Lambda(y,p^3q^3)$ and $\Lambda(z,p^3q^3)$ disjoint from each other.

Each positive entry (2,~3,~4,~5) in the table is a nonnegative integer coefficient linear combination of $2$ and $3$. Accordingly, for each $x\in\Z_{p^3q^2}$ such that $w^{72}_B(x)\neq 0$, we may define $A\cap \Lambda(x,p^3q^2)$ to be
 either just a single $2$-fiber, or a single $3$-fiber, or two $2$-fibers, or a $3$-fiber and a $2$-fiber, where each fiber is on scale $M$.
Furthermore, in those cases when $A\cap \Lambda(x,p^3q^2)$ consists of two fibers, we may place them in different $p^3q^3$-grids, guaranteeing that they do not overlap. Hence $A$ is a set, we have $A\equiv B$ mod 72, and, by Proposition \ref{cuboid},
$\Phi_{p^4q^4} \mid A$. Since $p,p^2,p^3,q,q^2$ all divide 72, $A$ inherits the property (\ref{e-table11}) from $B$. Hence $A$ satisfies all conclusions of the theorem.
\end{proof}

We remark that the same proof would also work for any $M=p^m q^n$ such that
$$\frac{D(M)}{p^3q^2}=\frac{p^{m-1}q^{n-1}}{p^3q^2}= p^{m-4}q^{n-3}\geq 2.$$
For example, we could take $M= p^5q^3$.

In Proposition \ref{prop:countex1} and \ref{prop:countex2}, we used $p=2$ and $q=3$. We now give a similar construction
for any pair of distinct odd primes $p$ and $q$. If one of the primes is 2, then the construction can be easily modified (and it is somewhat simpler), but we omit the details. We will need two easy number-theoretic lemmas.

\begin{lemma}\label{lem:pq-is-the-sum-of-p-and-q}
Let $p,q$ be distinct primes. If $pq\le K\in \NN$, then there exist $s_K,r_K\in \NN_0$ such that $s_Kp+r_Kq=K$.
\end{lemma}

\begin{lemma}\label{lem:mod2^t}
Let $p,q$ be distinct odd primes. Then for every $\ell \in \NN$ there exist $a,b \in \NN$ such that $p^a \equiv q^b \equiv 1\bmod{2^\ell}$.
\end{lemma}

\begin{theorem}\label{thm:general-example}
Let $p,q$ be distinct odd primes. We define the parameters $k,a,b,n,m\in\NN$, in that order, so that the following hold.
\begin{itemize}
\item[(i)] Choose $k$ so that $2^k\geq pq+1$.
\item[(ii)] Choose $a,b$ so that the conclusion of Lemma \ref{lem:mod2^t} holds with $\ell=2k$. Thus, we have $p^a=C_1 4^k +1$ and $q^b=C_2 4^k +1$ for some $C_1,C_2\in\NN$.
\item[(iii)] Assume that $p^a>q^b$ (otherwise we interchange $p$ and $q$).
\item[(iv)] Define $N=p^aq^b$ and $M=p^nq^m$, with
$n>a$ and $m>b$ large enough so that  $D(M)/N \geq q^b$.
\end{itemize}
Then there is a set $A$ of size $|A|=N=p^a q^b$ that satisfies
\begin{equation}\label{e-tables20}
\Phi_p\Phi_{p^2}\cdots\Phi_{p^a}\Phi_q \Phi_{q^2}\cdots\Phi_{q^b} \Phi_M \mid A.
\end{equation}
\end{theorem}

\begin{proof}
As in the proof of Proposition  \ref{prop:countex2}, we first construct a multiset $B \in \mathcal{M}^+(\Z_{N})$
such that $|B|=N$,
\begin{equation}\label{e-tables12}
\Phi_p\Phi_{p^2}\cdots\Phi_{p^a}\Phi_q \Phi_{q^2}\cdots\Phi_{q^b} \mid B,
\end{equation}
and each entry $w^N_B(x)$ for $x\in\Z_N$ is either zero or large enough so that we can apply Lemma \ref{lem:pq-is-the-sum-of-p-and-q}. We then lift it to a set $A \subset \Z_M$ such that $A \equiv B\bmod N$ and, additionally, $\Phi_M \mid A$.

We write $\Z_{N}$ as $\Z_{p^a}\oplus \Z_{q^b}$. The set $B$ will be defined via a $q^b\times p^a$ matrix with entries $w^{N}_B(x_{ij})$, where $x_{ij}$ is the element of $\Z_{N}$ such that $x_{ij}\equiv i$ mod $q^b$ and $x_{ij} \equiv j$ mod $p^a$.
We start with some building blocks.
We define the square matrices $G$ and $H$ so that $G$ is a $2^k\times 2^k$ matrix with all entries equal to $2^k$, and $H$ is a $2^{2k}\times 2^{2k}$ matrix with blocks equal to $G$ along the diagonal and zeroes everywhere else. We use $0$ to denote zero matrices as needed.

\[
G=
    \underbrace{
    \begin{array}{|ccc|}
        \hline
        2^k &  \cdots & 2^k \\
        \vdots & \ddots & \vdots \\
        2^k & \cdots & 2^k \\
        \hline
    \end{array}
    }_{2^k}
\ ,\quad\quad
H=
    \underbrace{
    \begin{array}{|c|c|c|c|}
        \hline
        G  & 0& \cdots & 0 \\
        \hline
        0& G & \cdots & 0\\
        \hline
        \vdots & \vdots & \ddots & \vdots \\
        \hline
     0 & 0  & \cdots & G \\
        \hline
    \end{array}
    }_{2^{k}\text{ blocks}}
\]
Note that, in both $G$ and $H$, each row and each column adds up to $2^k\cdot 2^k = 4^k$.

Let $Y$ be the matrix with $p^a=C_1 4^k +1$ columns and $q^b=C_2 4^k +1$ rows, defined as follows. We start with a matrix with $C_1 4^k $ columns and $C_2 4^k$ rows that consists of concatenated blocks equal to $H$. (The total number of such blocks is $C_1C_2$.) Then we add one row and one column where the entries are all equal to 1, as shown below. For $n\geq 1$, we use $1_n$ to denote the row vector $(1,1,\dots,1)$ with $n$ entries, and $1_n^t$ to denote its transpose.

\[
Y=
    \begin{array}{|c|c|c||c|}
        \hline
        H  &  \cdots & H & 1^t_{4^k} \\
        \hline
        \vdots  & \ddots & \vdots & \vdots \\
        \hline
     H  & \cdots & H & 1^t_{4^k}\\
        \hline\hline
       1_{4^k}  & \cdots &  1_{4^k} & 1    \\
       \hline
    \end{array}
\]

This matrix defines a multiset $B_0$ in $\ZZ_{N}$. Note that each row of $Y$  adds up to $C_1 4^k+1=p^a$, and
each column of $Y$  adds up to $C_2 4^k +1=q^b$. It follows that $|B_0|=p^aq^b=N$, and
\begin{equation}\label{e-tables10}
\Phi_p\Phi_{p^2}\cdots\Phi_{p^a}\Phi_q \Phi_{q^2}\cdots\Phi_{q^b}\mid B_0.
\end{equation}

We now define an ``adding a cuboid" operation\footnote{
In terms of mask polynomials, the operation in (\ref{add-a-cuboid}) corresponds to adding a polynomial of the form $X^c(1-X^{d_1 N/p^a})(1-X^{d_2N/q^d})$, with $d_1\in\{1,2,\dots,p^a-1\}$ and $d_2\in\{1,2,\dots,q^b-1\}$. This means that we are adding either an $N$-cuboid as defined in (\ref{def-delta}), or else an $N'$-cuboid for some $pq\mid N'\mid N$.
}
that preserves row and column sums. Given a $2\times 2$ submatrix of $Y$ with entries $a_1,a_2,a_3,a_4$, we may replace it by a submatrix with entries $a_1+1,a_2-1,a_3-1,a_4+1$ as shown below:
\begin{equation}\label{add-a-cuboid}
\begin{array}{c}
    \begin{array}{|c|c|c|}
        \hline
        a_1 &a_2 \\
        \hline
        a_3 & a_4 \\
        \hline
    \end{array}
    \quad
    \to
    \quad
    \begin{array}{|c|c|}
        \hline
        a_1+1 & a_2-1 \\
        \hline
        a_3 -1 & a_4+1 \\
        \hline
    \end{array}
    =\begin{array}{|c|c|}
        \hline
        a_1 & a_2 \\
        \hline
        a_3 & a_4 \\
        \hline
    \end{array}+\begin{array}{|c|c|}
        \hline
        1 & -1 \\
        \hline
        -1 & 1 \\
        \hline
    \end{array}
\end{array},
\end{equation}
and the matrix thus obtained has the same row sums and column sums as $Y$. The same remains true if we iterate a sequence of such operations or its inverses.

Our goal is to use the operation in (\ref{add-a-cuboid}) to get a new multiset $B\in\calm^+(\ZZ_{N})$ so that $|B|=|B_0|$, (\ref{e-tables10}) continues to hold for $B$, and additionally all nonzero entries in $B$ are large enough so that we could apply Lemma \ref{lem:pq-is-the-sum-of-p-and-q}. Specifically, we need to replace all the 1 entries by either 0 or a number greater than or equal to $pq$.

Recall that $p^a>q^b$, so that $C_1>C_2$. Then $Y$ has the block decomposition
\[
Y=
    \begin{array}{|c||c|c|c|c ||c|}
        \hline
        Y_0 & Y_1 & Y_2    &  \cdots & Y_{(C_1-C_2)2^k} &   1^t_{C_2 4^k} \\
        \hline\hline
       1_{C_2 4^k } & 1_{2^k} & 1_{2^k} &  \cdots &  1_{2^k} & 1    \\
       \hline
    \end{array},
\]
where $Y_0$ is a square matrix of size $C_2 4^k$, and each $Y_j$ matrix with $j\geq 1$ has $C_2 4^k$ rows and $2^k$ columns. We now apply the operation in (\ref{add-a-cuboid}) to the blocks shown above, indicating only those entries that are involved in the operation (all other entries will remain unchanged).

We first remove the 1 entries below $Y_0$ and to the right of it, by adding all cuboids whose top left vertex is on the diagonal of $Y_0$, bottom left vertex is in the last row of $Y$, top right vertex is in the last column of $Y$, and bottom right vertex is at the bottom right corner of $Y$. (There are $C_2 4^k$ such cuboids.) This is illustrated below.
\[
    \begin{array}{|c||c|}
        \hline
        Y_0 & 1^t_{C_2 4^k}  \\
      \hline
        \hline
        1_{C_2 4^k } & 1\\
        \hline
    \end{array}
=
        \begin{array}{|cccc||c|}
            \hline
                2^k & \cdots & \cdots & \cdots & 1 \\
                \cdots & 2^k & \cdots &\cdots & 1\\
                \vdots & \vdots &  \ddots &\vdots & \vdots \\
                \cdots & \cdots & \cdots & 2^k & 1\\
                \hline\hline
                1 & 1 & \cdots & 1 & 1\\
                \hline
            \end{array}
\quad \to \quad
        \begin{array}{|cccc||c|}
            \hline
                2^k+1 & \cdots & \cdots & \cdots & 0 \\
                \cdots & 2^k+1 & \cdots &\cdots & 0\\
                \vdots & \vdots &  \ddots &\vdots & \vdots \\
                \cdots & \cdots & \cdots & 2^k+1 & 0\\
                \hline\hline
                0 & 0 & \cdots & 0 & C_2 4^k + 1\\
                \hline
            \end{array}
\]
At this point, we have replaced $Y$ with the matrix
\[
    \begin{array}{|c||c|c|c|c ||c|}
        \hline
        Y'_0 & Y_1 & Y_2    &  \cdots & Y_{(C_1-C_2)2^k} &   0 \\
        \hline\hline
       0 & 1_{2^k} & 1_{2^k} &  \cdots &  1_{2^k} & C_2 4^k + 1   \\
       \hline
    \end{array},
\]
where all entries of $Y'_0$ are either zero or at least $2^k$. We still need to remove all the 1 entries below the blocks $Y_j$ with $j\geq 1$. Pick one such block, and note that it contains at least one square subblock equal to $G$. We will only work with that subblock and the row of 1's at the bottom; all other entries will remain unchanged. We add all cuboids whose top left vertex is on the diagonal of that block except for the last diagonal entry, bottom left vertex is in the last row of $Y$, and top right vertex is in the last column of $Y_j$. (There are $2^k -1$ such cuboids.) This is shown below.

\[
    \begin{array}{|c|}
        \hline
        Y_j \\
        \hline\hline
       1_{2^k} \\
       \hline
    \end{array}
=
\begin{array}{|ccccc|}
\hline
\vdots & \vdots & \vdots & \vdots & \vdots \\
        \hline \hline
        2^k & \cdots    &  \cdots & \cdots & 2^k  \\
         \cdots & 2^k      &  \cdots  &\cdots  &2^k  \\
         \vdots & \vdots  & \ddots & \vdots & \vdots  \\
         \cdots  & \cdots   &  \cdots &  2^k & 2^k  \\
         \cdots  & \cdots   &  \cdots &  \cdots & 2^k  \\
                                   \hline\hline
      \vdots  & \vdots & \vdots & \vdots & \vdots \\
        \hline \hline
        1 & 1 & \cdots  & 1 & 1
         \\
       \hline
    \end{array}
\quad\to\quad
  \begin{array}{|ccccc|}
\hline
\vdots & \vdots & \vdots & \vdots & \vdots \\
        \hline \hline
        2^k +1 & \cdots &\cdots      &  \cdots & 2^k-1  \\
         \cdots & 2^k +1 &\cdots       &  \cdots & 2^k -1 \\
         \vdots & \vdots  & \ddots & \vdots &\vdots \\
         \cdots & \cdots   & 2^k +1    &  \cdots & 2^k -1 \\
         \cdots & \cdots   & \cdots     &  \cdots & 2^k \\
                          \hline\hline
      \vdots & \vdots & \vdots & \vdots & \vdots \\
        \hline \hline
        0 & 0 & 0 & \cdots & 2^k
         \\
       \hline
    \end{array}
\]

Applying this procedure to all blocks $Y_j$, we get a matrix $Y'$ with row sums $p^a$,
column sums $q^b$, and all entries either zero or at least $2^k-1$. Further, all entries of $Y'$ are bounded by $C_24^k+1=q^b$.
The corresponding multiset $B\in\calm^+(\ZZ_{N})$
has total weight $|B|=N$ and satisfies (\ref{e-tables12}).

We now construct a set $A\subset\ZZ_M$ that inherits the above properties from $B$ and, additionally, satisfies $\Phi_M\mid A$.
As in the proof of Proposition \ref{prop:countex2}, let $\pi: \Z_{M}\to  \Z_{N}$ be the natural projection defined by $\pi(x)=x\bmod N$. For each $x\in\Z_{N}$, its preimage $\pi^{-1}(x)=\Lambda(x,N)$ is a union of $D(M)/N$ pairwise disjoint grids $\Lambda(y,D(M))$ with $y\in\pi^{-1}(x)$.

By the assumption (i), each nonzero entry in $B$ is greater than or equal to $pq$. By Lemma \ref{lem:pq-is-the-sum-of-p-and-q}, we can write each such entry as a linear combination of $p$ and $q$ with coefficients in $\NN_0$; note that the sum of these coefficients is bounded trivially by the size of the largest entry in $Y'$, that is, by $q^b$. Thus, for each $x\in\Z_N$ such that $w^N_B(x)\neq 0$, we may define $A\cap\Lambda(x,N)$ as a union of at most $q^b$ fibers in the $p$ and $q$ direction. By the assumption (iv) of the theorem, we have $D(M)/N>q^b$, so that we may place each such fiber in a different $D(M)$-grid, ensuring that $A$ is a set. By Proposition \ref{cuboid}, it follows that $\Phi_M\mid A$. Furthermore, since $A\equiv B$ mod $N$, we have $|A|=|B|=N$ and (\ref{e-tables20}) is inherited from (\ref{e-tables12}).
Hence, $A$ satisfies all requirements of the theorem.
\end{proof}

We note that the cuboid-adding operation (\ref{add-a-cuboid}) was also the key to the constructions in Propositions \ref{prop:countex1} and \ref{prop:countex2}. The matrices in both of these propositions can be obtained, by adding and subtracting cuboids, from a ``uniform" matrix with all entries equal to 1.

We did not try to optimize the exponents $m$ and $n$ in Theorem~\ref{thm:general-example}, opting instead for clarity of the presentation. For example, we could have used (\ref{add-a-cuboid}) more efficiently to get $Y'$ with all nonzero entries of size about $2^k$, reducing the size of $m$ and $n$.

The construction in Theorem~\ref{thm:general-example} admits two natural directions of generalization. One is that we could
modify it to guarantee \emph{simultaneous divisibility} by a block of the form
\[
\prod_{L : L_0 \mid L \mid M} \Phi_L(X),
\]
where \( L_0 = p^\alpha q^\beta \mid M = p^n q^m \). This can be achieved by replacing the single \( p \)- and \( q \)-fibers with long \( p^{n - \alpha + 1} \)- and \( q^{m - \beta + 1} \)-fibers, using Proposition~\ref{lma:longfiberslincom}. The argument remains valid in this setting, as long as the associated multiplicative constraints are satisfied and the disjointness of the fiber structure is maintained.

The construction also extends inductively to the case of \emph{arbitrary finite sets of primes} $\{p_1, \dots, p_r\}$, provided that the parameters involved are chosen sufficiently large. In particular, one may invoke a multivariate analogue of Lemma~\ref{lem:mod2^t} to guarantee the necessary congruences modulo $2^\ell$, and apply the additive decomposition as in Lemma~\ref{lem:pq-is-the-sum-of-p-and-q} recursively. The key structural requirements (separability of fibers, modular compatibility, and size bounds) scale in a controlled way as the number of primes increases. However, this would complicate the construction considerably, and in any case simpler examples (such as Example \ref{ex-3primes1scale}) are available when 3 or more prime factors are allowed.


\section{Lower bounds under the (T2) assumption}
\label{T2-section}

\subsection{Structure results under the (T2) assumption}
\label{sec-lowerT2-basic}

Assume that $M = \prod_{i=1}^{K} p_i^{n_i}$ and that $K \geq 2$.
In this section we prove structure results for sets $A\subset\ZZ_M$ obeying the conditions (T1) and (T2).
We first prove the short counting lemma mentioned in the introduction.

\begin{lemma}\label{prime-power-disjoint}
\cite[Lemma 2.1]{CM}
Assume that $A\oplus B=\ZZ_M$. Let $S_A^*$ be the set of prime powers
$p^\alpha$ such that $\Phi_{p^\alpha}(X)$ divides $A(X)$ and let $S_M^*$ be the set of all prime powers that divide $M$. Then
$$
|A|=\prod_{s\in S_A^*} \Phi_s(1),\ |B|=\prod_{s\in S_B^*} \Phi_s(1).
$$
Moreover, the sets $S_A^*$ and $S_B^*$ are disjoint, and $S_A^*\cup S_B^*=S_M^*$.
\end{lemma}

\begin{proof}
The equality $S_A^*\cup S_B^*=S_M^*$ follows from (\ref{mask-e2}). Furthermore, we have $\Phi_s(1)=p$ if $s=p^\alpha$ is a prime power, and $\Phi_s(1)=1$ otherwise. This implies the divisibility chain
$$
M=  \prod_{s\in S_M^*} \Phi_s(1)
\ \Big| \  \prod_{s\in S_A^*} \Phi_s(1)  \prod_{s\in S_B^*} \Phi_s(1)
\ \Big| \  A(1)B(1)=M.
$$
Hence equality must hold at each step, and the lemma follows.
\end{proof}

Recall that the truncation $A' \in \calm^+ (\mathbb{Z}_{M'})$ of a multiset $A\subset\calm^+(\ZZ_M)$ relative to a set of divisors $S\subset\cald(M)$ was introduced in Proposition \ref{prop:truncation} and defined formally in \eqref{eq:truncationweigthdefn}. The next lemma says that if $A$ obeys (T2), then its truncation $A'$ is uniformly distributed mod $M'$.

\begin{lemma}\label{lma:T2maximaltruncation}
Let $A \in \calm^+ (\mathbb{Z}_M)$ and set $S = S_A^*$, defined in Lemma \ref{prime-power-disjoint} . Let $A' \in \calm^+ (\mathbb{Z}_{M'})$ be the truncation of $A$ relative to $S$. If $A$ satisfies $(T2)$, then
\begin{equation}\label{eq:maximaldivisibilityforuniformT2}
1 + X + \cdots + X^{M' - 1} \, \big\vert \, A'(X).
\end{equation}
\end{lemma}

\begin{proof}
    If $A$ satisfies $(T2)$, then $\Phi_s (X) \mid A' (X)$ for every $1 \neq s \mid M'$ by applying Proposition \ref{prop:truncation} (ii) for $S=S_A^*$. Since
    $$
    1 + X + \cdots + X^{M' -1} = \prod_{1 \neq s \mid M'} \Phi_s (X),
    $$
    the divisibility \eqref{eq:maximaldivisibilityforuniformT2} of $A' (X)$ then follows.
\end{proof}

\subsection{A diagonal argument}
\label{sec-lowerT2-diagonal}

Our first result carries no restrictions on the number of prime factors of $M$.
For each $i \in \{1,\ldots,K\}$, we use the notation
$$
\beta_i:=\max\{\beta\in\NN:\ \Phi_{p_i^\beta}\mid A\}.
$$

\begin{proposition}\label{prop:T2topdown}
Assume that $M : = \prod_{i=1}^{K} p_i^{n_i}$ for $K \geq 2$, where $p_1,\dots,p_K$ are distinct primes.
Suppose that $A \in \calm^+ (\mathbb{Z}_M)$ satisfies $(T2)$. Then, if there exists some $N = p_1^{\gamma_1} \cdots p_K^{\gamma_K}$ such that $\Phi_N (X) \mid A(X)$ and $\gamma_i > \beta_i$ for all $i \in \{1,\ldots,K\}$, then
$$
A(1) \geq \min(p_1, \cdots, p_K) \prod_{s \in S_A^*} \Phi_s (1).
$$
\end{proposition}

In particular, this gives a negative answer to Question \ref{Q2} in the introduction under the additional assumption that the additional unsupported divisor $\Phi_N$ is as indicated in the proposition.

\begin{proof}
Choose $A' \in \calm^+ (\mathbb{Z}_{M'})$ to be the truncation of $A$ relative to $S = S_A^* \cup \{N\}$. Notice, then, that $(A' \equiv A'' \mod M'')$, where $A'' \in \calm^+ (\mathbb{Z}_{M''})$ is the truncation of $A$ relative to $S_A^*$. This is illustrated in Figure \ref{fig:maximalN} (where $M'$ and $M''$ are both labeled, for reference).

Proposition \ref{prop:truncation} combined with Lemma \ref{lma:T2maximaltruncation} then imply that
\begin{equation}\label{eq:completegrid}
A'(X) = w \big(1 + X + \cdots + X^{M'' - 1}\big) \mod X^{M''} - 1,
\end{equation}
for some integer weight $w \geq 1$.   We claim that $w \geq \min (p_1,\ldots,p_K)$.  This follows because, for each $x \in \mathbb{Z}_{M'}$,  we have
$$
w_{A}^{M''} (x) : = \sum\limits_{ \{\bar{x} \in \mathbb{Z}_N : \bar{x} \, \equiv \, x \textrm{ mod } M'' \}} w_A^{M'} (\bar{x}).
$$
However, since $D(M') = M''$, we see that  $\{\bar{x} \in \mathbb{Z}_{M'} : \bar{x} \, \equiv \, x \textrm{ mod } M'' \} = \Lambda (x, D(M'))$.  Since
$$
\Phi_{M'} (X) \mid A(X) \Leftrightarrow \Phi_{M'} (X) \mid (A \cap \Lambda (x,D(M'))),
$$
for each $x \in \mathbb{Z}_{M'}$,  the bound \eqref{e-lamleung} of Lam and Leung gives $w \geq \min (p_1,\ldots,p_K)$.  This, together with \eqref{eq:completegrid} and the fact that
$
 \prod_{s \in S_A^*} \Phi_s (1) = M'',
$
gives the result.
\end{proof}


\begin{figure}[h!]
    \centering
    \includegraphics[width=0.4\linewidth]{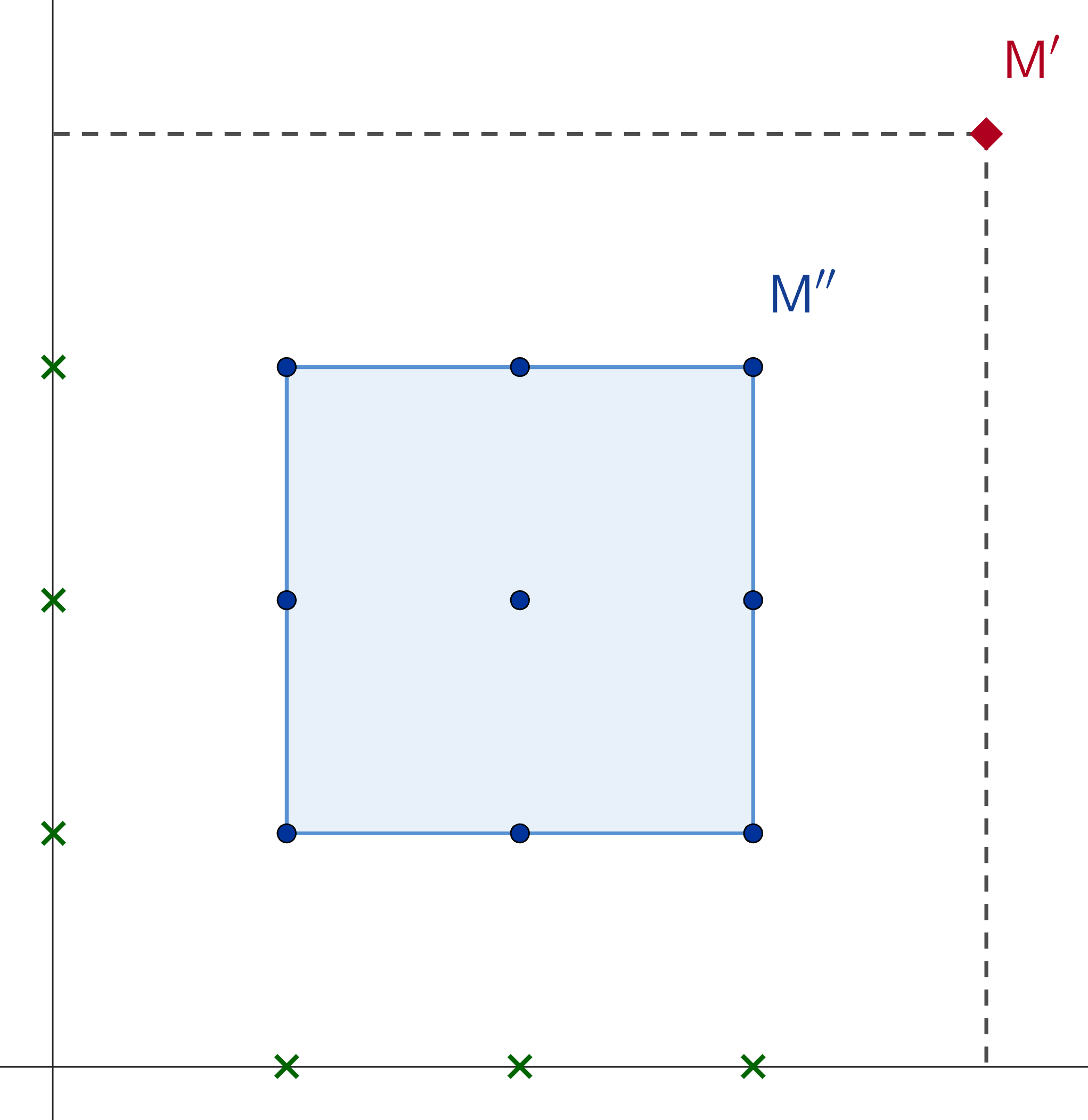}
    \caption{When $N$ is maximal, the truncation of $A$ relative to $S_A^* \cup \{N\}$ has a full block of cyclotomic divisors below an unsupported divisor $\Phi_{M'} (X)$.}
    \label{fig:maximalN}
\end{figure}

\subsection{(T2) lower bounds for two prime factors}
\label{sec-lowerT2-2primes}

In this subsection, we will assume that $M$ has two distinct prime factors $p_1$ and $p_2$. As in Section \ref{sec-4primeexample}, we will abbreviate $p:=p_1$ and $q:=p_2$. We will continue to use the numerical indices where appropriate, so that for example $F^N_1$ will still denote a fiber in the $p$ direction on scale $N$ and $\capexp_1(S)$ will denote the set of exponents of $p=p_1$ in $S$.

Under these assumptions upon $M$, we then have the following general size increase when $A$ satisfies $(T2)$ and admits an unsupported divisor $N$.

\begin{theorem}\label{thm:additionaldivisor}
    Let $M = p^m q^n$ and suppose that $A \in \calm^+ (\mathbb{Z}_M)$ satisfies $(T2)$. If there exists some $N = p^{\gamma} q^{\eta}$ such that $\Phi_N (X) \mid A(X)$ and $\Phi_{p^{\gamma}} (X), \Phi_{q^{\eta}} (X) \nmid A(X)$, then
\begin{equation}\label{eq:additionaldivisorincrease}
A(1) > \prod_{s \in S_A^*} \Phi_s (1).
\end{equation}
\end{theorem}
Of course, we always have that
$$
A(1) \geq \prod_{s \in S_A^*} \Phi_s (1),
$$
for any multiset $A \in \calm^+ (\mathbb{Z}_M)$. Theorem \ref{thm:additionaldivisor} is an improvement because it shows that this inequality must be strict if $A$ admits an unsupported divisor.

We will prove Theorem \ref{thm:additionaldivisor} in four cases depending upon the location of the unsupported divisor. Proposition \ref{prop:T2topdown} already handles the case where $N$ has maximal $p_1$ and $p_2$ exponents. The remaining three cases are the content of Proposition \ref{prop:Niscentral} and Corollary \ref{cor:remainingedgecasesN}. Of the remaining cases to consider, the situation where both $\gamma$ and $\eta$ are neither maximal nor minimal contains the most new ideas. We present the proof of this result first, often referencing the key ideas which are developed in the proof of later cases.

To this end, let
$$
\alpha_i : = \min \{\alpha \in \mathbb{N} : \Phi_{p_i^{\alpha}} (X) \mid A(X) \}
$$
denote the minimal prime power exponents associated to cyclotomic divisors of $A(X)$. We will also use the notation
$$
\textsf{EXP}^*(i) : = \{\upsilon \in \mathbb{N} : \Phi_{p_i^\upsilon} (X) \mid A(X) \}
$$
to denote the exponents associated to prime power cyclotomic divisors of $A$.

\begin{proposition}\label{prop:Niscentral}
    Let $M = p^m q^n$ and $A \in \calm^+ (\mathbb{Z}_{p^m q^n})$. Suppose that $A$ satisfies $(T2)$ and also has an unsupported divisor $\Phi_N (X) \mid A(X)$ with $N = p^{\gamma} q^{\eta}$ satisfying $\alpha_1 <\gamma < \beta_1$ and $\alpha_2 <\eta < \beta_2$. Then $A$ has the size increase given in \eqref{eq:additionaldivisorincrease}.
\end{proposition}

\begin{proof}
Using Proposition \ref{prop:truncation}, we assume that
$$
\textsf{EXP}^*(1) : = \{1,\ldots,\gamma - 1, \gamma + 1,\ldots,m\}, \quad \textsf{EXP}^*(2) : = \{1,\ldots,\eta - 1, \eta + 1,\ldots,n\}.
$$
In this case, inequality \eqref{eq:additionaldivisorincrease} becomes $A(1) > p^{m-1} q^{n-1}$. We assume, in contradiction, that $A(1) = p^{m-1} q^{n-1}$. This configuration of cyclotomic divisors is illustrated in Figure \ref{fig:megablocks}. Let
\begin{align*}
& M_1 = p^{\gamma - 1} q^{\eta - 1},& \, & M_2 = p^{\gamma -1 } q^{n},& \, &M_3 = p^{m} q^{\eta -1},& \, &M_4 = p^m q^n& \\[1ex]
& N_1  = p q,& \, &N_2 = p q^{\eta +1 },& \, &N_3 = p^{\gamma + 1} q^{n},& \, &N_4 = p^{\gamma + 1} q^{\eta + 1}&
\end{align*}
so that each $M_l$ (resp. $N_l$) denotes the upper-right (resp., lower left) vertex of the block $\mathcal{B}_l$ which are shown in Figure \ref{fig:megablocks}. Notice that, as $A$ satisfies $(T2)$ (and has been uniformized), our blocks are the clusters
$
\mathcal{B}_l : = \{s \in \mathcal{D}(M) : N_i \mid s \mid M_i \}.
$
We carefully demarcate when we work in each of the clusters $\mathcal{B}_1,\ldots,\mathcal{B}_4$, as the corresponding steps are repeated in proofs of later results.

\noindent
\underline{\it Block $\mathcal{B}_4$}: Applying Proposition \ref{lma:longfiberslincom} to the block $\mathcal{B}_4$ produces polynomials $P_1(X),P_2(X) \in \mathbb{Z}[M]$ with non-negative coefficients such that
\begin{equation}\label{eq:M4longfibers}
A (X) = P_1 (X) F_{1,m-\gamma}^{M_4} (X) + P_2 (X) F_{2,n-\eta}^{M_4} (X) \mod X^{M_4}- 1.
\end{equation}
We may assume that $P_1 \not\equiv 0$, and so we will work with the block $\mathcal{B}_2$ in the next step (this is where we rely on the fact that both $\mathcal{B}_2$ and $\mathcal{B}_3$ are contained in $S_A$).

\begin{figure}[h!]
    \centering
    \begin{subfigure}[t]{0.45\textwidth}
        \centering
        \includegraphics[width=.7\textwidth]{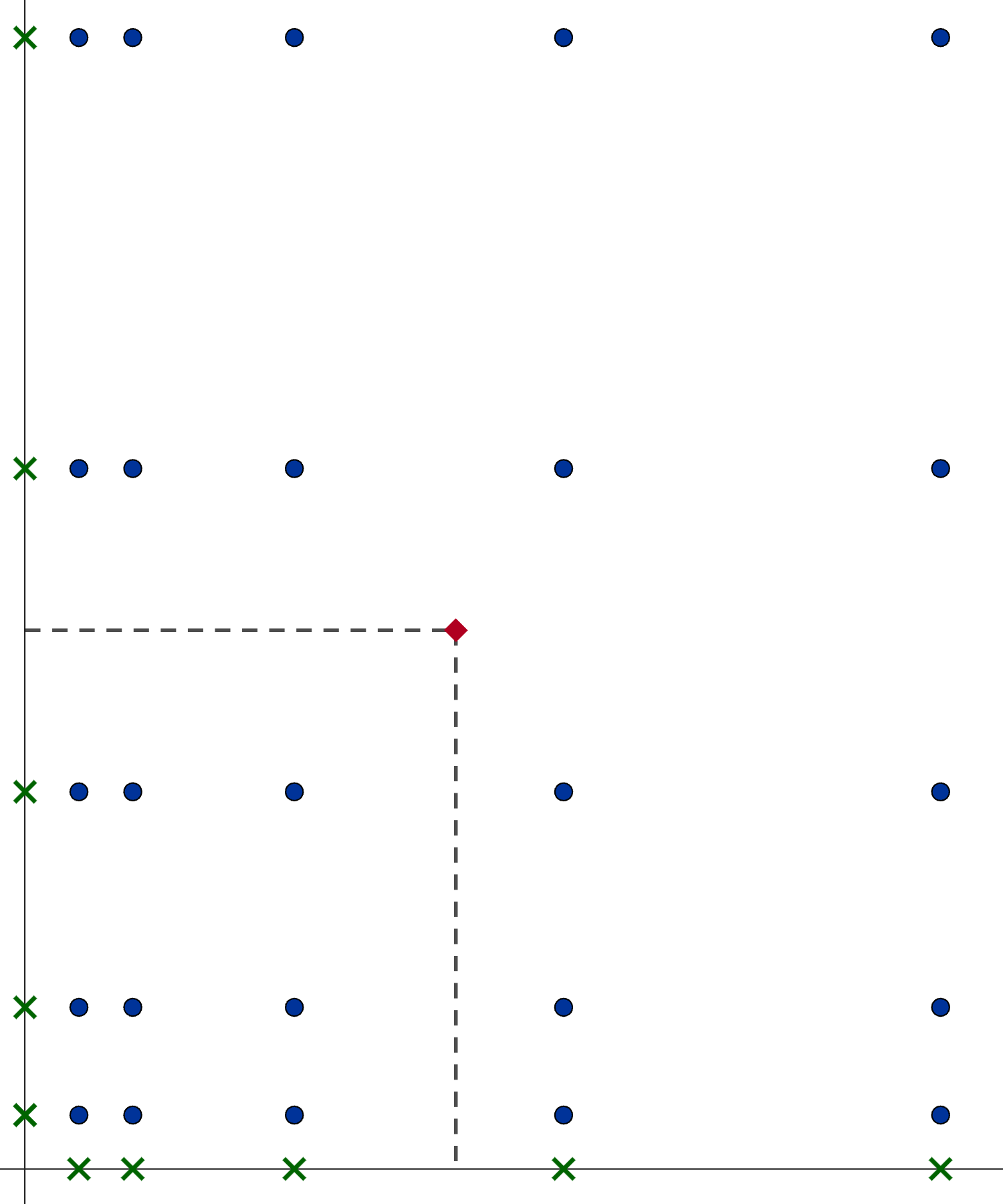}
        \caption{Before truncation}
    \end{subfigure}%
    \begin{subfigure}[t]{0.45\textwidth}
        \centering
        \includegraphics[width=.88\textwidth]{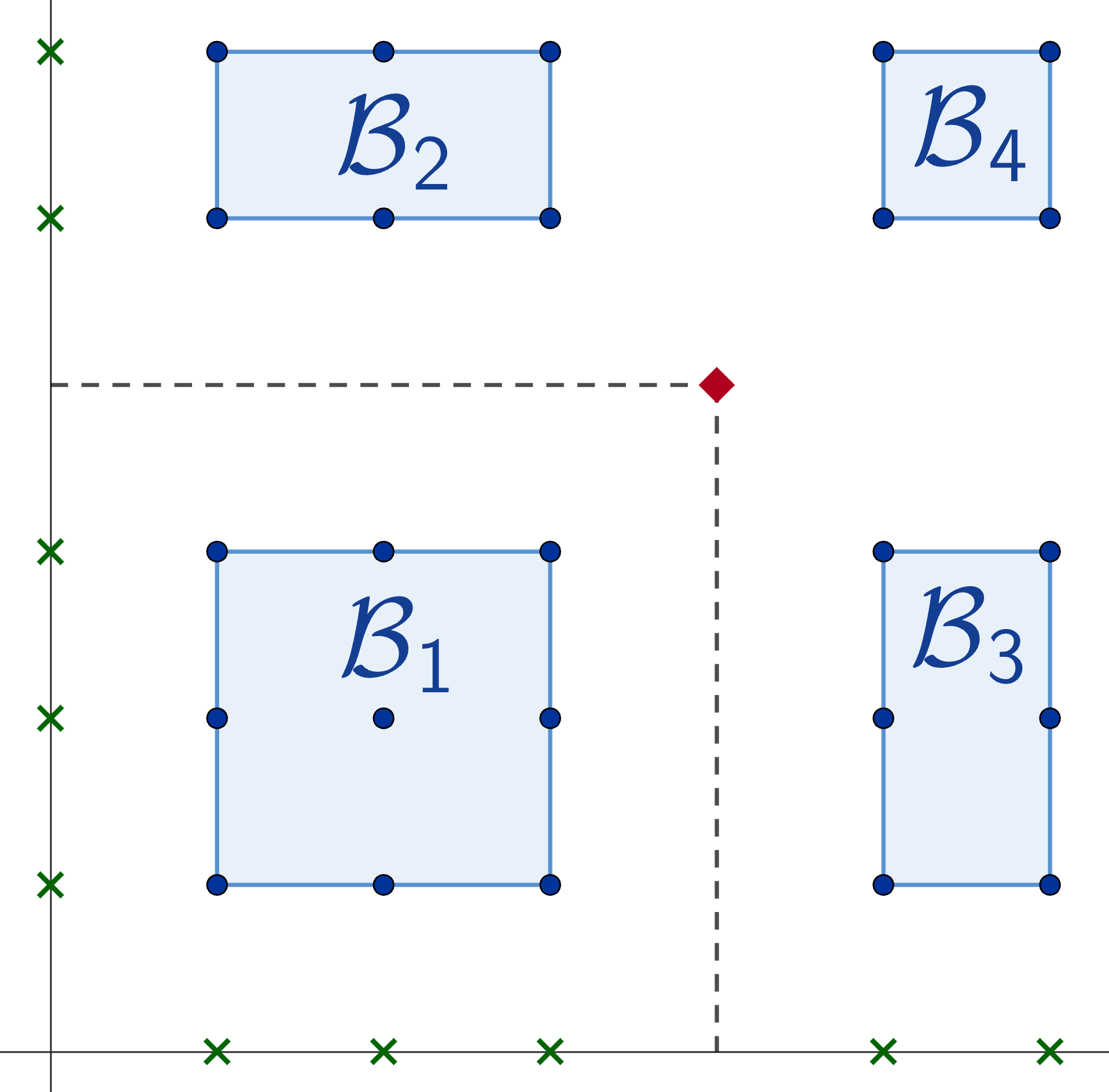}
        \caption{After truncation}
	 \end{subfigure}
    \caption{Green cross points are prime power divisors, blue circle points are $(T2)$ divisors, and the red diamond point is the unsupported divisor. After truncation, we have four complete blocks of divisors.}
     \label{fig:megablocks}
\end{figure}

\noindent
\underline{\it Block $\mathcal{B}_2$ or $\mathcal{B}_3$}: Our next goal is to show that $A \textrm{ mod } M_2$ can be expressed as a linear combination of long fibers in the $q$ direction alone. Applying Proposition \ref{lma:longfiberslincom} to the block $\mathcal{B}_2$ produces polynomials $Q_1 (X),Q_2(X) \in \mathbb{Z}[X]$ with non-negative coefficients such that
\begin{equation}\label{eq:M2longfibers}
A(X) = Q_1 (X) F_{1,\gamma -1}^{M_2}(X) + Q_2 (X) F_{2,n - \eta}^{M_2} (X) \mod X^{M_2}-1.
\end{equation}
We want to show that we can take $Q_1 \equiv 0$ in this long fiber decomposition.

We now make use of the fact that $\Phi_{q^{\eta +1}} (X),\ldots ,\Phi_{q^n} (X) \mid A(X)$. All $q^{n - \eta}$-fibers at scale $M_2$ are necessarily constant mod $q^{s}$ for every $s \in \{ \eta +1,\ldots ,n\}$. Hence, the number of $p^{\gamma - 1}$-fibers at scale $M_2$ must be equidistributed along these same $q^s$ cosets. However, the length of each such $p^{\gamma-1}$-fiber at scale $M_2$ is maximal in $\mathbb{Z}_{M_2}$, which follows from the fact that $p^{\gamma - 1} \mid \mid M_2$. Hence, equidistribution in number is equivalent to being a complete grid (at least, in this special case where our long fibers have maximal length). This lets us then assume that $Q_1 \equiv 0$ in \eqref{eq:M2longfibers}.

Recall that we are assuming that $A \textrm{ mod } M_4$ has mask polynomial \eqref{eq:M4longfibers} with $P_1 \not\equiv 0$. This means that $A \textrm{ mod } M_4$ has (at least) one $p^{m - \gamma}$ long fiber in the $p$ direction with a positive weight.  This long fiber collapses to a point $y \in \mathbb{Z}_{M_2}$ with multiplicity at least $p^{m - \gamma}$ (which just means that $w_A^{M_2} (y) \geq p^{m - \gamma}$). Since we have assumed that $A \textrm{ mod } M_2$ is a linear combination of $q^{n-\eta}$-fibers, there exists a set $F \in \calm^+ (\mathbb{Z}_{M_4})$ such that
$$
F(X) = X^y F_{2, n - \eta}^{M_2}(X) \mod X^{M_2} - 1,
$$
and such that
\begin{equation}\label{eq:M2fiberweight}
(A\cap F)(X):= w \, F_{2,n - \eta}^{M_2} (X) \mod X^{M_2} - 1 \Rightarrow (A \cap F)(1) =w\,q^{n - \eta},
\end{equation}
where the weight $w \in \mathbb{N}$ is constant and satisfies $w \geq p^{m - \gamma}$. We remind the reader that the multiset $A \cap F \in \calm^+ (\mathbb{Z}_{M})$ is defined via the equality of weights
$$
w_{A \cap F}^M (x) = w_A^M (x) w_F^M (x), \quad  \forall x \in \mathbb{Z}_M.
$$

\noindent
\underline{\it Block $\mathcal{B}_1$}: The multiset $A \cap F$ collapses to a single point $x \in \mathbb{Z}_{M_1}$ which satisfies $w_A^{M_1} (x)\geq p^{m - \gamma} q^{n - \eta}$. Observe that, by Lemma \ref{lma:T2maximaltruncation} and the assumption that $A(1) = p^{m-1} q^{n - 1}$, one has
\begin{equation}\label{eq:balancedweightM1}
A(X) = p^{m - \gamma} q^{n - \eta} \big(1 + X + \cdots + X^{M_1} \big) \mod X^{M_1} - 1.
\end{equation}
Hence, we must have that
\begin{equation}\label{eq:weightequality}
(A \cap F) (1) = w_A^{M_1} (x) = w \, q^{n - \eta} = p^{m - \gamma} q^{n - \eta} \Rightarrow w = p^{m - \gamma}.
\end{equation}
Consequently, we know that
\begin{equation}\label{eq:fiberwithconstantweight}
(A \cap F)(X) = p^{m - \gamma} F_{2, n - \eta}^{M_2} (X) \mod X^{M_2} - 1.
\end{equation}
We obtain that $A\cap F \mod X^{M_4} - 1$ is a union of $p^{m-\gamma}$-fibers.

\noindent
\underline{\it The unsupported divisor $N$}: In fact, we now want to work with the set $A \cap F \textrm{ mod } N$, in addition to considering the whole set $A \textrm{ mod } N$. This which amounts to only examining the part of $A$ which intersected the long fiber as in \eqref{eq:fiberwithconstantweight}. Hence, let $B \in \calm^+ (\mathbb{Z}_N)$ satisfy $B \equiv (A \cap F) \textrm{ mod } N$.

Since $\Phi_N (X) \mid A(X)$, there exist polynomials $R_1 (X), R_2 (X) \in \mathbb{Z}[X]$ with non-negative coefficients such that
\begin{equation}\label{eq:unsupportedforB}
A(X) = R_1 (X) F_1^N (X) + R_2 (X) F_2^N (X) \mod X^N - 1.
\end{equation}
We also remark that \eqref{eq:fiberwithconstantweight} implies that
\begin{equation}\label{eq:heavypointforB}
B(X) = p^{m - \gamma} q^{n - \eta} X^z \mod X^{N/p} - 1
\end{equation}
for some $z \in \mathbb{Z}_{N/p}$.

We now rely upon the fact that $P_1 \not\equiv 0$ in \eqref{eq:M4longfibers}. Any $p^{m-\gamma}$-fiber at scale $M_4$ collapses to a single point in $\mathbb{Z}_N$ with multiplicity $p^{m - \gamma}$.
As we have seen earlier, the preimage of multiset $B \textrm{ mod } N $ taken $\textrm{ mod } M_4$ is actually a union of $p^{m-\gamma}$-fibers. Hence, $p^{m-\gamma}\mid w_B^{N}(z')$ for every $z'\in \mathbb{Z}_N$.
We distinguish two cases.

In case where $(z' * F_2^N) \subset A \textrm{ mod } N$ (i.e., $z'$ belongs to a $q$-fiber at scale $N$), the congruence for $A(X)$ implies

$$ w_A^{M_1}(z')= w_A^{N/p}(z')=w_B^{N/p}(z')=p^{m-\gamma}q^{n-\eta},$$
which further implies that
$$w_A^{N/p}(z'+\frac{jN}{pq})=0 \textrm{ and } w_A^{N}(z'+\frac{jN}{q})=0, \ \ \textrm{ for }j\not\equiv 0 \textrm{ mod } q.$$
Therefore, $w_A^{M_1}(z')\ge p^{m-\gamma}q^{n-\eta+1}$, contradiction.

Otherwise, if $z'\in\mathbb{Z}_N$ belongs to a single  $p$-fiber at scale $N$, then $w_B^{N/p}(z')=pw_B^N(z')$, which is divisible by $p^{m-\gamma+1}$, contradiction.
\end{proof}

\begin{remark}\label{rmk:lessoverhead}
\rm{
Notice that the conditions $\alpha_i <\gamma_i < \beta_i$ for $i \in \{1,2\}$ guarantee that all of the blocks $\mathcal{B}_1,\mathcal{B}_2, \mathcal{B}_3,\mathcal{B}_4$ which were utilized in the proof of Proposition \ref{prop:Niscentral} are non-empty. However, if $\gamma_i \not\in [\alpha_i, \beta_i]$ for some $i \in \{1,2\}$, then a corresponding combination of blocks will be empty. One such example is illustrated in Figure \ref{fig:onemaximalexponentN} below. Even in these cases, however, the proof of Proposition \ref{prop:Niscentral} applies---with the caveat that one must omit any steps which correspond to empty blocks. This gives the following Corollary.
}
\end{remark}

\begin{corollary}\label{cor:remainingedgecasesN}
Let $M = p^m q^n$ and $A \in \calm^+ (\mathbb{Z}_{p^m q^n})$. Suppose that $A$ satisfies $(T2)$ and also has an unsupported divisor $\Phi_N (X) \mid A(X)$ with $N = p^{\gamma_1} q^{\gamma_2}$ with $\gamma_i \not\in [\alpha_i, \beta_i]$ for some $i \in \{1,2\}$. Then, $A$ has the size increase \eqref{eq:additionaldivisorincrease}.
\end{corollary}

\begin{proof}
    Left to the interested reader (see Remark \ref{rmk:lessoverhead}).
\end{proof}

\begin{figure}[h!]
    \centering
\includegraphics[width=0.3\linewidth]{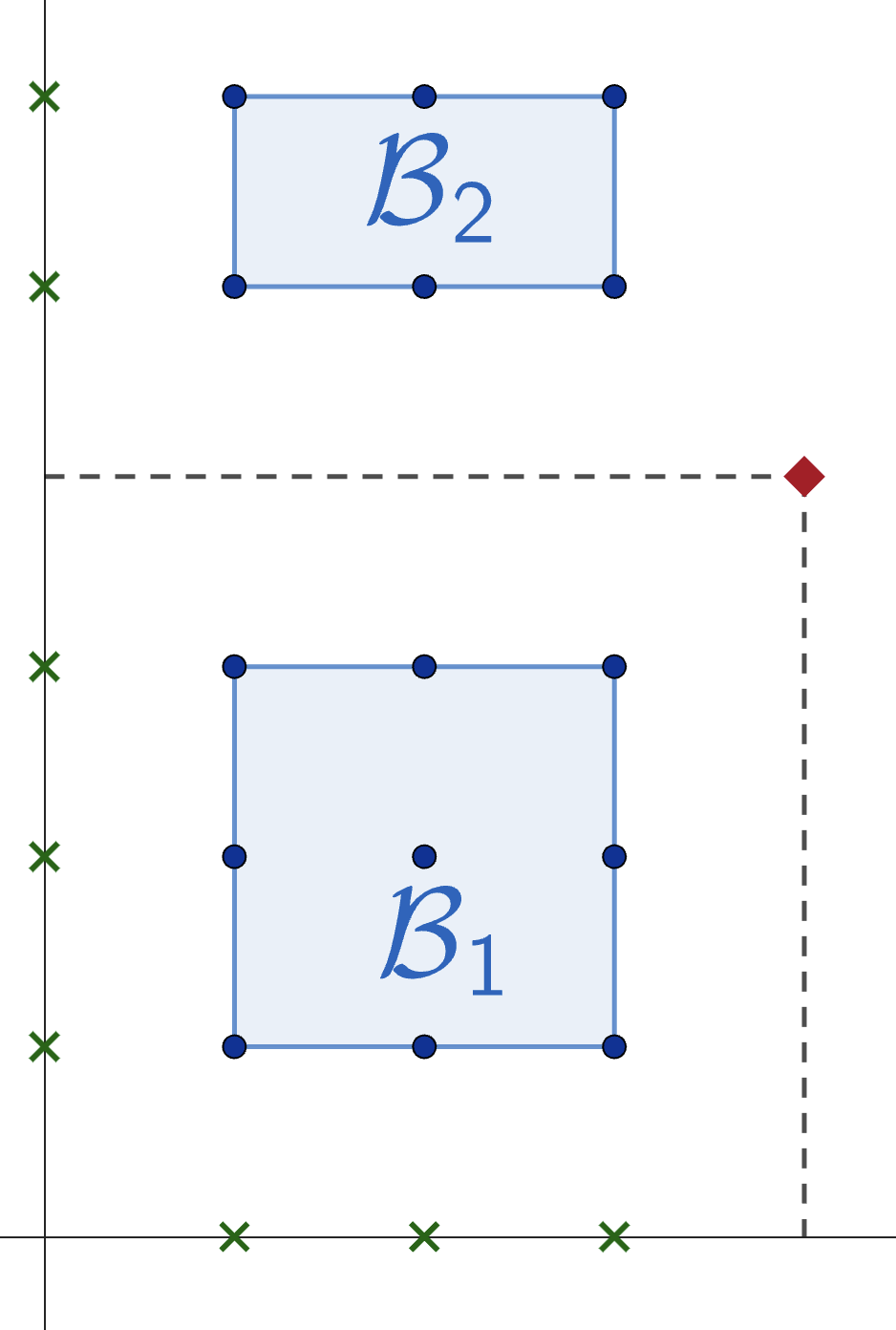}
    \caption{An example of a configuration of cyclotomic divisors where $N = p^{\gamma_1} q^{\gamma_2}$ satisfies $\gamma_1 \not\in [\alpha_1, \beta_1]$. Notice that, in this scenario, the blocks $\mathcal{B}_3$ and $\mathcal{B}_4$ are necessarily empty. As before: green cross points are prime power divisors, blue circle points are $(T2)$ divisors and the red diamond point is an unsupported divisor.}
    \label{fig:onemaximalexponentN}
\end{figure}

Notice that Proposition \ref{thm:additionaldivisor} furnishes a negative answer to Question 2 from Section \ref{sec-intro} in the following special case.

\begin{corollary}\label{cor:unsupported}
    Suppose that $A\subset\nno$ satisfies (T1) and (T2), and that $\textrm{lcm} (S_A) = p^m q^n$ for two distinct prime factors $p,q$. Then $A_0$ does not have any unsupported divisors.
\end{corollary}




\section{An example with 4 prime factors}
\label{sec-4primeexample}


\subsection{An example}
\label{subsec-4primes-example}

In this section, we prove Theorem \ref{thm-T2-smallset} and give an affirmative answer to Question \ref{Q2} from Section \ref{CM-subsection}. Theorem \ref{thm-T2-4primes} describes our example in more detail. While the set $A$ below is constructed in a cyclic group $\ZZ_M$, we can clearly map it to a subset of $\nno$.

\begin{theorem}\label{thm-T2-4primes}
Let $M=(p_1p_2p_3p_4)^4$, where
\begin{equation}\label{4p-e1}
p_1>40\hbox{ and }p_i<p_{i+1}<2p_i\hbox{ for }i=1,2,3.
\end{equation}
(The existence of primes satisfying (\ref{4p-e1}) is guaranteed by Bertrand's Postulate.)
Then there exists a set $A\subset \ZZ_M$ such that:
\begin{itemize}
\item[(i)] the prime power cyclotomic divisors of $A(X)$ are $\Phi_{p_i^\alpha}(X)$ for all $i=1,2,3,4$ and $\alpha=2,3,4$,
\item[(ii)] $A$ satisfies both (T1) and (T2), so that in particular we have $|A|=(p_1p_2p_3p_4)^3$,
\item[(iii)] additionally, $A(X)$ has the unsupported cyclotomic divisor $\Phi_{p_1p_2p_3p_4}(X)$.
\end{itemize}
\end{theorem}

\begin{proof}
Let $N=p_1p_2p_3p_4$ for short, and recall that
$$
F_{i,3}(X)=\frac{X^M-1}{X^{M/p_i^3}-1} \hbox{ for }i=1,2,3,4.
$$
Define a set $U\subset\ZZ_M$ via
$$
U(X)=\prod_{i=1}^4 F_{i,3}(X),
$$
so that $U=N\ZZ_M \simeq \ZZ_{D(M)}$.
By Lemma \ref{lma:longfibersdivisors}, $F_{i,3}$ is divisible by all $\Phi_s$ such that $p_i^2\mid s \mid M$, since then we have $s \mid M$ but $s\nmid (M/p_i^3)$.
It follows that $U$ has the prime power cyclotomic divisors indicated in (i), and satisfies both (T1) and (T2).
In the terminology of \cite{LL1}, $U$ is a standard tiling set with these cyclotomic divisors. We also have $U\oplus B=\ZZ_M$, where
\begin{equation}\label{4p-e2}
B(X)=\prod_{i=1}^4 (1+X^{M_i}+\dots+X^{(p_i-1)M_i}).
\end{equation}

We do not have $\Phi_N \mid U$, since $U$ modulo $N$ is simply the point $0$ with weight $|U|=N^3$. However, we will now rearrange the long fibers (i.e., sets with mask polynomial $X^lF_{i,3}(X)$ for some $l$) in $U$ to get a set $A$ whose mask polynomial does have $\Phi_N$ as an additional divisor.

We first construct a decomposition
$$U=U_1\cup U_2\cup U_3\cup U_4,
$$
where the sets $U_i$ are nonempty, pairwise disjoint, and each $U_i$ is a union of long fibers in the $p_i$ direction. Specifically, let $d_1,d_2,d_3$ be integers such that
\begin{equation}\label{4p-e3}
1\leq d_i<p_i^3-1,\ i=1,2,3.
\end{equation}
For any $x\in U$, we can represent its array coordinates $(x_1,x_2,x_3,x_4)$ as $x_i=p_i\bar x_i$, We then define
$$
U_i:= Q_i*F_{i,3},\ \ i=1,2,3,4,
$$
where
\begin{equation}\label{4primes-defU}
\begin{split}
Q_1&:=\{(0,x_2,x_3,x_4)\in U:\ d_3\leq \bar x_2<p_2^3,\ 0\leq \bar x_3<d_1,\ \bar x_4\hbox{ arbitrary}\},
\\
Q_2&:=\{(x_1,0,x_3,x_4)\in U :\ 0\leq \bar x_1<d_2,\ d_1\leq \bar x_3<p_3^3,\ \bar x_4\hbox{ arbitrary}\},
\\
Q_3&:=\{(x_1,x_2,0,x_4)\in U:\ d_2\leq \bar x_1<p_1^3 ,\ 0\leq \bar x_2<d_3,\ \bar x_4\hbox{ arbitrary}\},
\\
Q_4&:=\{(x_1,x_2,x_3,0)\in U:\ 0\leq \bar x_1<d_2,\ 0\leq \bar x_2<d_3,\ 0\leq \bar x_3<d_1\}
\\
&\cup
\{(x_1,x_2,x_3,0)\in U:\ d_2\leq \bar x_1<p_1^3,\ d_3\leq \bar x_2<p_2^3,\ d_1\leq \bar x_3<p_3^3\}.
\end{split}
\end{equation}
We check that each $x\in U$ belongs to exactly one of the sets $U_i$. The proof is as follows. Let $x\in U$.

\begin{itemize}
\item Suppose first that $\bar x_1<d_2$. If $\bar x_3\geq d_1$, then $x\in U_2$.
 If $\bar x_3< d_1$ and $\bar x_2\geq d_3$, then $x\in U_1$.
  If $\bar x_3< d_1$ and $\bar x_2< d_3$, then $x\in U_4$.
These choices are unique.
\smallskip

\item Assume now that $\bar x_1\geq d_2$. If $\bar x_2 <d_3$, then $x\in U_3$.
If $\bar x_2 \geq d_3$ and $\bar x_3<d_1$, then $x\in U_1$.
If $\bar x_2 \geq d_3$ and $\bar x_3\geq d_1$, then $x\in U_4$. These choices, again, are unique.
\end{itemize}

Next, we claim that the parameters $d_1,d_2,d_3$ may be chosen so that
\begin{equation}\label{4p-e4}
p_i\mid |Q_i|\hbox{ for }i=1,2,3,4.
\end{equation}
We already have $p_i^3\mid |U_i|$ by construction (since $p_i^3\mid |F_{i,3}|$); by (\ref{4p-e4}), we actually have $p_i^4\mid |U_i|$ for each $i$.

To ensure (\ref{4p-e4}), we let $d_1:=p_1p_4$, $d_2:=kp_2$ for some $k\in\NN$ to be chosen later, and $d_3:=p_3$. Then:
\begin{itemize}
\item $|Q_1|=d_1(p_2^3-d_3)=p_1p_4(p_2^3-d_3)p_4^3$ is divisible by $p_1$,

\smallskip
\item $|Q_2|=d_2(p_3^3-d_1)=kp_2 (p_3^3-d_1)p_4^3$ is divisible by $p_2$,

\smallskip
\item $|Q_3|=d_3(p_1^3-d_2)=p_3(p_1^3-d_2)p_4^3$ is divisible by $p_3$.
\end{itemize}

We now consider $Q_4$, with
$$
|Q_4|=d_1d_2d_3+ (p_1^3-d_2)(p_2^3-d_3) (p_3^3-d_1).
$$
We have $p_4\mid d_1\mid d_1d_2d_3$. To ensure that $p_4$ also divides the second term in $|Q_4|$, it suffices to choose $k$ so that $p_4$ divides $p_1^3-d_2$. The numbers $p_1^3-kp_2$ with $k=1,2,\dots,p_4$ all have distinct residues mod $p_4$, so that there exists a value of $k$ such that $p_4$ divides $p_1^3-kp_2$, as claimed.

It remains to check that the above values of $d_1$ and $d_2$ are permissible, in the sense that they obey (\ref{4p-e3}). It suffices to ensure that
$$
p_1p_4<p_3^3-1\hbox{ and }p_4p_2< p_1^3-1.
$$
This follows if we verify that $p_ip_j < p_1^3-1 $ for any choice of distinct indices $i,j$ since $p_1^3-1 \le p_k^3-1$ for every $k$. By (\ref{4p-e1}), we have
$$
p_ip_j\leq p_3p_4<(4p_1)(8p_1)=32p_1^2<p_1^3-1,
$$
as claimed.

We are now ready to rearrange our initial set $U$ to produce $A$. For each $i=1,2,3,4$, we divide $Q_i$ into $p_i$ pairwise disjoint subsets of cardinality $|Q_i|/p_i$ each:
$$
Q_i=Q_{i,1}\cup\dots \cup Q_{i,p_i},\  \ |Q_{i,1}|=\dots=|Q_{i,p_i}|=|Q_i|/p_i.
$$
We then let
$$
A_i:= \bigcup_{j=1}^{p_i}\bigcup_{x\in Q_{i,j}}(x+jM_i)*F_{i,3}.
$$
In other words, if $x*F_{i,3}$ is a long fiber in $U_i$, we shift that fiber in the $p_i$ direction (consistently with the direction of $F_{i,3}$) by an increment of $jM_i$, where $j$ is chosen based on which set $Q_{i,j}$ contains $x$. This is similar to the ``fiber shifting'' constructions of Szab\'o \cite{Sz} (see also \cite{LL2, LL-even}).

Let $A=A_1\cup A_2\cup A_3\cup A_4$.
We prove that $A$ is a set. To do so, it suffices to verify that $A_1,A_2,A_3,A_4$ are pairwise disjoint. Let $a\in A_i$ and $a'\in A_j$ for some $i\neq j$. Then $a=u+\nu M_i$ and $a'=u'+\nu'M_j$ for some $u\in U_i$, $u'\in U_j$, $\nu\in\{1,\dots,p_i\}$, and $\nu'\in\{1,\dots,p_j\}$. We claim that
\begin{equation}\label{4primes-e60}
\exists k\not\in\{i,j\} \hbox{ such that }u_k\neq u'_k.
\end{equation}
Then $a_k=u_k\neq u'_k=a'_k$, so that $a\neq a'$ as claimed.

The proof of (\ref{4primes-e60}) is by direct case-by-case verification based on (\ref{4primes-defU}).
\begin{center}
\begin{equation*}
\begin{tabular}{ | m{1.5cm} || m{1 cm}| m{1cm} | m{1.5 cm} | m{1cm} | m{1.5 cm} | m{1.5 cm} |}
\hline
$i,j$ & 1,2 & 1,3& 1,4& 2,3 & 2,4& 3,4  \\
  \hline
$k$ & 3 & 2 & 2 or 3 & 1 & 1 or 3& 1 or 2  \\
  \hline
\end{tabular}
\end{equation*}
\end{center}
In the third and last two cases, the value of $k$ depends on whether $(u'_1,u'_2,u'_3,0)$ belongs to the first or second set in the definition of $Q_4$ in (\ref{4primes-defU}).

Since $|A_i|=|U_i|$ for each $i$, it follows that $|A|=|U|=N^3$, where we recall that $N=p_1p_2p_3p_4$. We now check that $\Phi_N\mid A$, where $N=p_1p_2p_3p_4$. It suffices to check that $\Phi_N\mid A_i$ for each $i=1,2,3,4$. Fix such $i$. For each $x\in Q_{i,j}$,
we have $x\equiv 0$ and $F_{i,3}(X)\equiv 0$ mod $N$, so that
the fiber $(x+jM_i)*F_{i,3}$ reduced modulo $N$ is simply the point $j'N/p_i$ with multiplicity $p_i^3$, where $j'\equiv j(N/p_i)^3 \pmod{p_i}$. Thus $A_i$ reduced modulo $N$ is the fiber $F^N_i$ with multiplicity $p_i^{-1}|Q_i|p_i^3=|Q_i|p_i^2$. It follows that $A_i$ is divisible by $\Phi_N$ as required.

We now present two different methods of verification that $A$ has the cyclotomic divisors claimed in parts (i) and (ii) of the theorem. Each method provides a different insight as to whether a similar construction could also furnish a counterexample to the Coven-Meyerowitz conjecture; we discuss this in more detail after the proof of the theorem is completed.

\medskip
\noindent
{\bf Method 1: Divisor sets.} We note that the set $B$ defined in (\ref{4p-e2}) is a {\em standard tiling complement} in the terminology of \cite{LL1}. If we can prove that
\begin{equation}\label{4p-e66}
A\oplus B=\ZZ_M,
\end{equation}
it follows that $\Phi_s\mid A$ for all $s \mid M$ such that $s\neq 1$ and $\Phi_s \nmid B$. This includes all $\Phi_{p_i^\alpha}(X)$ for all $i=1,2,3,4$ and $\alpha=2,3,4$ (by Lemma \ref{prime-power-disjoint}, $A$ cannot have any other prime power cyclotomic divisors), as well as all cyclotomic divisors required by (T2) (see \cite[Proposition 3.4]{LL1}).

By Sands's Theorem \cite{Sands},
(\ref{4p-e66}) will follow if we prove that $\Div(A)\cap\Div(B)=\{M\}$, where
$$
\Div(A)=\{(a-a',M):\ a,a'\in A\}
$$
and similarly for $B$. Let us write $d=\prod_{i=1}^4 p_i^{\delta_i} $, where $\delta_i \in \{0,1,2,3,4\}$  for a divisor of $M$.
We have
$$
\Div(B)=\{d=\prod_{i=1}^4 p_i^{\delta_i} \in\ZZ_M:\ \forall i\in\{1,2,3,4\}, \hbox{ either }\delta_i=0\hbox{ or } \delta_i=4\}.
$$
In other words $d_i=0$ or $p_i \nmid d_i$ if we write $d=\sum_{i=1}^4d_iM_i$.
It therefore suffices to prove that
\begin{equation}\label{4p-divA}
\Div(A)\subset\{M\}\cup\{d\in\ZZ_M:\ \exists k\in\{1,2,3,4\}\hbox{ such that } \delta_k \notin\{0,4\}
\}.
\end{equation}
Indeed, suppose that $a,a'\in A$ satisfy $a\neq a'$. As before, we write $a=u+\nu M_i$ and $a'=u'+\nu'M_j$ for some $u\in U_i$, $u'\in U_j$, $\nu\in\{1,\dots,p_i\}$, and $\nu'\in\{1,\dots,p_j\}$. We need to prove that $(a-a',M)$ belongs to the set on the right-hand side of (\ref{4p-divA}). We consider the following cases.
\begin{itemize}
\item If $i\neq j$, we choose $k$ as in (\ref{4primes-e60}); since $p_k$ divides both $u_k$ and $u'_k$,
the claim in (\ref{4p-divA}) is true with this value of $k$.
\item If $i=j$ and there exists $k\neq i$ such that $u_k\neq u'_k$, then the claim is true for this $k$.
\item If $i=j$ and $u_\ell=u'_\ell$ for all $\ell\neq i$, we must have $\nu\neq \nu'$, so that the claim is true with $k=i$.

\end{itemize}

\medskip
\noindent
{\bf Method 2: Direct verification.}
We need to prove that $\Phi_s\mid A$ for all $s=\prod_{j\in J}p_j^{\alpha_j}$, where $J\subset\{1,2,3,4\}$, $J\neq\emptyset$, and $\alpha_j\in\{2,3,4\}$ for all $j\in J$. We fix such $s$ for the rest of the proof.

For all $i\in J$, we have $\Phi_s\mid F_{i,3}$ by Lemma \ref{lma:longfibersdivisors}. It follows that
\begin{equation}\label{4p-e90}
\Phi_s\mid U_i\hbox{ and } \Phi_s\mid A_i\hbox{ for all }i\in J.
\end{equation}
Further, since $\Phi_s\mid U$ as noted above, we also have
\begin{equation}\label{4p-e91}
\Phi_s(X) \ \Big| \ U(X)-\sum_{i\in J}U_i(X) = \sum_{i\not\in J} U_i(X).
\end{equation}
Let $L=\prod_{j\in J}p_j^4$. Then $s\mid L$, so that for any polynomial $G(X)$ we have $\Phi_s\mid G$ if and only if $\Phi_s\mid (G\bmod L)$. In particular, $\Phi_s$ divides $\sum_{i\not\in J} U_i(X)$ mod $X^L-1$. However, we have $U_i\equiv A_i$ mod $L$ for all $i\not\in J$, so that
$$
\sum_{i\not\in J} U_i(X)\equiv  \sum_{i\not\in J} A_i(X) \mod X^L-1.
$$
It follows from (\ref{4p-e91}) that $\Phi_s\mid \sum_{i\not\in J} A_i(X)$. This together with (\ref{4p-e90}) implies that $\Phi_s\mid A$ as claimed.
\end{proof}

\subsection{Further remarks}
\label{subsec-4primes-discussion}

We have not tried to optimize the size of $M$ or the conditions on the size of the primes, so that improvements in that regard may be possible. On the other hand, no similar construction based on shifting fibers can work when $M$ has only three distinct prime factors, since the three-prime analogue of $U$ does not admit a decomposition similar to (\ref{4primes-defU}). We do not know whether the answer to Question \ref{Q2} in $\ZZ_M$ is positive or negative in that case.

We do not know whether some modification of the construction in Theorem \ref{thm-T2-4primes} could be used to give a counterexample to the Coven-Meyerowitz conjecture. We do not see an obvious reason why this could not ultimately work, but there are also significant obstacles, which we now discuss.

We adopt the notation from the proof of Theorem \ref{thm-T2-4primes}. As explained in the introduction, one could try to construct a tiling $A\oplus B'=\ZZ_M$, where $B'$ is not divisible by $\Phi_N$. It is easy to find sets $B'$ that satisfy some (but not all) of the requirements to be a tiling complement for $A$. For instance, let $B'$ be any set in $\ZZ_M$ such that
$$
B'(X)\equiv \frac{X^N-1}{X-1} + \prod_{i=1}^4 (X^{N/p_i}-1) \mod X^N-1.
$$
Then $\Phi_s\mid B'$ for all $s$ such that $s\in\cald(N)\setminus\{N\}$ (in particular $\Phi_{p_i}\mid B'$ for all $i\in\{1,2,3,4\}$), but $\Phi_N\nmid B'$. This, however, is not sufficient to produce a tiling.

An argument similar to that in the Method 2 part of the proof of Theorem \ref{thm-T2-4primes} shows that $\Phi_s\mid A$ for $s=\prod_i p_i^{\alpha_i}$, where $\alpha_i=1$ for exactly one value of $i$. However,
this leaves out all $\Phi_s$ with $\alpha_i=1$ for two or three values of $i$. We do not see how to ensure that, in addition to all of the above properties, at least one of $A$ or $B'$ has those divisors.

Alternatively, one could consider the divisor sets of $A$ and $B'$. In constructions like the one above (possibly with minor modifications of the parameters), we expect that $\Div(A)$ will occupy most of the set in (\ref{4p-divA}). The only elements of that set that we know to {\it not} belong to $\Div(A)$ are those $d$ for which $p_i\mid d_i\neq 0$ for one value of $i$, and $p_j\nmid d_j$ for all $j\neq i$.
This does not leave much room to construct a tiling complement that is substantively different from the standard tiling set $B$ (or its easy modifications such as dilates).

We cannot exclude the possibility that, starting from a similar construction for $A$ but with more scales or distinct prime factors, alternative tiling complements not satisfying (T2) could in fact be found.


\section*{Acknowledgement}
We express our gratitude to the anonymous referee for their comments and suggestions, which significantly improved the clarity and precision of the paper. In particular, we are thankful for the simplification at the end of the proof of Proposition 9.4.

The research was partly carried out at the Erd\H os Center, R\'enyi Institute, in the framework of the semester``Fourier analysis and additive problems".

The first author was supported by Hungarian National Foundation for Scientific Research NKFIH, Grants STARTING 150576, FK 142993, Excellence 154121 and by the J\'anos Bolyai Research Fellowship of the Hungarian
Academy of Sciences.
The second author was supported by NSERC Discovery Grant 22R80520.
The third author was supported by NSERC Discovery Grants 22R80520 and GR010263.
The fourth author was supported by Hungarian National Foundation for Scientific Research, Grants
OTKA K138596 and STARTING Grant 150576 and ARC Discovery Project DP250104965


\noindent
		{\sc Gergely Kiss:}\\
        Corvinus University of Budapest, Department of Mathematics \\
		Fővám tér 13-15, Budapest 1093, Hungary,\\
        and\\
		HUN-REN Alfr\'ed R\'enyi Mathematical Institute\\
		Re\'altanoda utca 13-15, H-1053, Budapest, Hungary\\
		E-mail: {\tt kiss.gergely@renyi.hu}

\medskip

\noindent{{\sc Izabella {\L}aba:}  \\
			The University of British Columbia, Department of Mathematics \\ 1984 Mathematics Road, Vancouver, Canada, V6T 1Z2}, \\
		E-mail: {\tt ilaba@math.ubc.ca}

\medskip

\noindent{{\sc Caleb Marshall:}  \\
			The University of British Columbia, Department of Mathematics \\ 1984 Mathematics Road, Vancouver, Canada, V6T 1Z2}, \\
		E-mail: {\tt cmarshall@math.ubc.ca}

 \medskip

\noindent{{\sc G\'abor Somlai:}  \\
			E\"otv\"os Lor\'and University, Institute of Mathematics, Algebra and Number Theory
            Department \\ P\'azm\'any P\'eter s\'et\'any 1/C, Budapest, Hungary, H-1117}, \\
            and\\
            HUN-REN Alfr\'ed R\'enyi Mathematical Institute\\
		Re\'altanoda utca 13-15, H-1053, Budapest, Hungary\\
		E-mail: {\tt gabor.somlai@ttk.elte.hu and gabor.somlai@unimelb.edu.au}

\end{document}